\documentclass[letterpaper,english, 12pt]{article}
\usepackage{babel} 
\usepackage[ansinew]{inputenc}
\usepackage{amsmath, amssymb, amsthm}
\usepackage{nicefrac}
\usepackage[arrow, matrix, curve]{xy}
\usepackage{cite}
\usepackage{graphics}
\usepackage{graphicx}
\usepackage[lmargin=3.00cm,rmargin=3.00cm,tmargin=3cm,bmargin=3cm]{geometry}

\newcommand{\R}{\mathbb{R}}
\newcommand{\N}{\mathbb{N}}
\newcommand{\Z}{\mathbb{Z}}
\newcommand{\Q}{\mathbb{Q}}
\newcommand{\C}{\mathbb{C}}

\newcommand{\F}{\mathcal{F}}

\renewcommand{\H}{\mathcal{H}}
\newcommand{\curlyC}{\mathcal{C}}
\newcommand{\const}{\textnormal{const}}
\newcommand{\id}{\textnormal{id}}

\newcommand{\im}{\textnormal{im}}

\newcommand{\colim}{\textnormal{colim}}

\newcommand{\into}{\hookrightarrow}
\newcommand{\tr}{\textnormal{tr}}

\newcommand{\Map}{\textnormal{Map}}
\newcommand{\rel}{\textnormal{ rel }}
\newcommand{\Diff}{\textnormal{Diff}}
\renewcommand{\1}{{\bf 1}}
\newcommand{\hcob}{\textit{h}-cobordism }
\newcommand{\hcobs}{\textit{h}-cobordisms }

\renewcommand{\P}{\mathcal{P}}

\renewcommand{\Re}{\textnormal{Re}}
\renewcommand{\tilde}{\widetilde}
\renewcommand{\bar}{\overline}

\newtheorem{Lemma}{Lemma}[section]
\newtheorem{Theorem}[Lemma]{Theorem}
\newtheorem*{Theorem*}{Theorem}
\newtheorem{Proposition}[Lemma]{Proposition}
\newtheorem{Corollary}[Lemma]{Corollary}

\theoremstyle{definition}

\newtheorem{Axiom}[Lemma]{Axiom}

\newtheorem{Definition}[Lemma]{Definition}

\theoremstyle{remark}
\newtheorem{Remark}[Lemma]{Remark}

\date{\today}
\begin{document}

\title{Axioms for Higher Twisted Torsion Invariants}
\author{Christopher Ohrt}
\maketitle

\begin{abstract}
This paper attempts to investigate the space of various characteristic classes for smooth manifold bundles with local system on the total space inducing a finite holonomy covering. These classes are known as twisted higher torsion classes. We will give a system of axioms that we require these cohomology classes to satisfy. Higher Franz Reidemeister torsion and twisted versions of the higher Miller-Morita-Mumford classes will satisfy these axioms. We will show that the space of twisted torsion invariants is two dimensional or one dimensional depending on the torsion degree and spanned by these two classes. The proof will greatly depend on results on the equivariant Hatcher constructions developed in \cite{Hatcherconstruction}.
\end{abstract} 

\tableofcontents

\section{Introduction}

Higher torsion invariants have been developed by J. Wagoner, J. R. Klein, K. Igusa, M. Bismut, J. Lott, W. Dwyer, M. Weiss, E. B. Williams, S. Goette and many others (\cite{Igusa2}, \cite{0793.19002}, \cite{0837.58028}, \cite{1077.19002}, \cite{1071.58025}).\\
In his paper \cite{Igusa1} K. Igusa defined a higher torsion invariant to be a characteristic class $\tau(E)\in H^{4k}(B;\R)$ of a  smooth bundle $E\to B$ satisfying an additivity and a transfer axiom (see section 2 of \cite{Igusa1}). He proved that the set of higher torsion invariants forms a two dimensional vector space spanned by the higher Reidemeister torsion and the Miller-Morita-Mumford classes.\\
But higher Reidemeister torsion or Igusa Klein torsion can be defined in a more general way: It is a characteristic class $\tau^{IK}(E,\rho)\in H^{2k}(B;\R)$ for a  smooth bundle with an unitary representation $\rho:\pi_1 E\to U(m)$ factorizing through a finite group (See for example \cite{Igusa2}). For our purposes it will be better to look at finite complex local systems on $E$ instead. After a choice of a base point, this corresponds to a representation of the fundamental group as can be found for example in T. Szamuely's book \cite{1189.14002}. Regarding that, we will define a twisted higher torsion invariant to be a characteristic class $\tau(E;\F)\in H^{2k}(B;\R)$ for a finite local complex system $\F$ on $E$ inducing a finite holonomy covering satisfying six axioms: The first two are versions of the original two axioms for non-twisted torsion invariants, which will respect the local system. The next axiom will guarantee that the non-twisted torsion invariant $\tau(E; \1)$ obtained from a twisted torsion invariant by inserting the constant local system will be zero in degree $4l+2$, since there is no non-twisted torsion in these degrees. The remaining three axioms will determine the dependence of the torsion class on the local system. The third of these will be very specific and might be dropped based on a conjecture of Milnor (in \cite{Milnor}) we will discuss in section 5.\\
The goal of this paper is to show an analogous result to Igusa's on twisted torsion invariants. For this we will generalize Igusa's paper \cite{Igusa1} step by step: In the second chapter, we will define twisted higher torsion invariants. \\
In the third section, we will repeat why the Igusa Klein torsion $\tau^{IK}$ satisfies the axioms and introduce a twisted version of the Miller-Morita-Mumford classes $M^{2k}$ and show that these also satisfy the axioms. The MMM classes will be zero in degree $4l+2.$ Then we will state our main theorem:



\begin{Theorem}[Main Theorem]
The space of higher twisted torsion invariants in degree $4l$ on bundels with simple fiber and rationally simply connected base is two dimensional and spanned by the twisted MMM class and the twisted Igusa-Klein torsion, and one dimensional in degree $4l+2$ and spanned by the Igusa-Klein torsion. Especially, for any twisted torsion invariant $\tau$, there exists a unique $a\in \R$ and a (not necessarily unique) $b\in \R$, so that
	$$\tau=a\tau^{IK}+bM.$$
\end{Theorem}

The scalars $a$ and $b$ can be calculated as follows: For torsion in degree $4l$ we look at the universal line bundle $\lambda: ES^1\to \C\P^\infty.$ Since the cohomology of $H^{2k}(\C\P^\infty;\R)$ is one dimensional, the torsion invariant of the associated $S^1$-bundle $S^1(\lambda)$ and the associated $S^2$-bundle $S^2(\lambda)$ over $\C\P^\infty$ will determine the scalars $a$ and $b$. In degree $4l+2$ we only have to calculate $a$ by looking at a fiberwise quotient $S^1(\lambda)/(\Z/n)$ of the $n$-action on $S^1$. This admits a non-trivial finite complex local system and therefore has a non-trivial higher twisted torsion.\\
Before we prove the main theorem, we will extend an higher twisted torsion invariant to have values on bundles with vertical boundaries and then define a relative torsion for  bundle pairs (defined in section 4), which we will use to deconstruct any bundle into easier pieces and keep control over the torsion.\\
In the fifth section, we will show that the main theorem holds on $S^1$-bundles.\\
Then we will define the difference torsion to be
	$$\tau^\delta:=\tau-a\tau^{IK}-bM$$
and we will see that $\tau^\delta=0$ for every sphere, disk bundle and odd dimensional lens space bundle. In \cite{Hatcherconstruction} we give an explicit base for the space of \hcob bundles of a lens space and the calculations in this paper show that the difference torsion will be zero on these basis elements. From this crucial observation we can deduce that the difference torsion will be a fiber homotopy invariant, and in section 7 we will show that this fiber homotopy invariant must be trivial if it is restricted to bundles with rationally simply connected base and simple fiber.\\
This paper is the product of my work with Kiyoshi Igusa during my stay at the Brandeis University, Waltham, MA in the academic year 2011/2012 and lead up to further work on the equivariant Hatcher constructions with Thomas Goodwillie and Kiyoshi Igusa in the academic year 2012/2013. I want to thank Kiyoshi for the great support, advice, and guidance he offered me. I also want to thank Ulrich Bunke from my home university in Regensburg, Germany, for the many comments and corrections he contributed.

\section{Axioms and Definitions}

\subsection{Preliminaries}
Throughout the whole paper, let $F\into E\stackrel{p}{\to} B$ be a smooth fiber bundle, where $E$ and $B$ are compact smooth manifolds, $p$ is a smooth submersion, and $F$ is a compact orientable $n$-dimensional manifold with or without boundary. In the boundary case, there is a subbundle $\partial F\to \partial^v E\to B$ of $E$. We call $\partial^vE$ the vertical boundary of $E$. We assume that $B$ is connected and that the action of $\pi_1B$ on $F$ preserves the orientation of $F$. We also assume that $\pi_1 B$ is finite, which immediatly implies that the bundle $E$ is unipotent (as required in \cite{Igusa1}).

These are all similiar assumptions to the ones for considering non-twisted higher torsion classes. Additionally to those, we assume that $E$ comes equipped with a finite complex local system $\F.$ By "`finite"' we mean that there exists a finite covering $\tilde E\to E$ such that the pull-back of the local system is trivializable. These local systems are sometimes also called hermitian local coefficient systems because they induce a well defined hermitian inner product on each fiber. We will often call $\F$ just local coefficient system.\\
Now we repeat another construction from Igusa's paper \cite{Igusa1}:\\
Let $T^vE$ denote the vertical tangent bundle of $E$. This is the subbundle of the tangent bundle $TE$ of $E$ consisting of all tangent vectors mapping to zero in $TB$, that is, $T^vE$ is the kernel of $Tp:TE\to TB$. The Euler class
	$$e(E)\in H^n(E;\Z)$$
of the bundle $E$ is defined to be the usual Euler class of $T^vE$.\\
The transfer on oriented bundles
	$$tr_B^E:H^*(E;\Z)\to H^*(B;\Z)$$
is given by
	$$tr_B^E(x)=p_*(x\cup e(E)),$$
where 
	$$p_*:H^{*+n}(E;\Z)\to H^*(B;\Z)$$
is the push-down operator or Umkehr map. Over $\R$, it is given as the composition of two maps
	$$H^{k+n}(E;\R)\to H^k(B;H^n(F;\R))\to H^k(B;\R)$$
where the first map comes from the Serre spectral sequence of the bundle and the second map is induced by the coefficient map $H^n(F;\Z)\to \Z,$ given by evaluating on the orientation class of the fiber. For details see \cite{0976.57026} or \cite{Igusa3}.\\
If the orientation of the fiber $F$ is reversed, both $e(E)$ and $p_*$ change sign. Thus, the transfer is independent of the choice of orientation of $F$. For the basic properties of the transfer, see \cite{0306.55017}. The main property that we need is that, for closed fibers $F$,
	$$tr_B^E=(-1)^ntr_B^E.$$
So, rationally, $tr_B^E=0$ if $n=\dim F$ is odd.

\subsection{Higher Twisted Torsion Invariants}

Now we are ready to give the definition of a twisted higher torsion invariant. Most of the axioms were proposed by K. Igusa in \cite{Igusa4}.

\begin{Definition} A higher twisted torsion invariant in degree $2k$ with $k\in \N$ is a rule $\tau_k$, which assigns to any bundle $F\into E\to B$ with closed fiber $F$ and local coefficient system $\F$ on $E$ a cohomology class $\tau_k(E,\F)\in H^{2k}(B;\R)$ subject to the following Axioms. We will drop the degree out of the notation most of the time and just write $\tau$.
\end{Definition}

\begin{Axiom}[Naturality]
$\tau_k$ is a characteristic class in degree $2k$. That means for a map $f:B'\to B$ and a bundle $F\into E\to B$ with local coefficient system $\F$ on $E$ we have 
	$$\tau_k(f^*(E),f^*\F)=f^*\tau(E,\F)\in H^{2k}(B';\R),$$
where $f^*$ denotes the pull-back along $f$.\\
\end{Axiom}

\begin{Remark} The naturality axiom immediately implies triviality on trivial bundles $\tau_k(B\times F,\F)=0$, if $\F=\1$ is the constant local system. Furthermore, if $B$ is simply connected, the local system $\F$ on $B\times F$ will pull-back from a local system $\F_F$ on $F$ under the projection $B\times F\to F.$ Now we can look at an $F$-bundle $E\to B'$ with local system $\F_E$ which induces the local system $\F_F$ on the fiber $F_*$ over the base point $*\in B'.$ If we pull back $E$ along the trivial map $\const_*:B \to B',$ we get the trivial bundle $B\times F\to B$ and the local system $\F_E$ will induce the local system $\F$ on $B\times F.$ Thereby we see that $\tau(B\times F,\F)=0$ for all local systems $\F$ on $B\times F$ as long as $B$ is simply connected.
\end{Remark}

Let $E_1$ and $E_2$ be bundles over $B$ with local coefficient systems $\F_1$ and $\F_2$, such that there is an isomorphism $\phi:\partial^v E_1\to\partial^v E_2\neq \emptyset$ and we have for the restrictions of the local systems
	$$\F_{|\partial^vE_1}\cong\phi^*\F_{|\partial^vE}.$$
Then we can glue them together to a local coefficient system $\F:=\F_1\cup_\phi \F_2$ on $E_1\cup_\phi E_2$. 

\begin{Axiom}[geometric additivity]
In the setting from above we have for any twisted torsion invariant $\tau$
	$$\tau(E_1\cup_\phi E_2,\F)=\frac 1 2 (\tau(DE_1,\F_1^l\cup_\id\F_1^r)+\tau(DE_2,\F_2^l\cup_\id\F_2^r)),$$
where $DE_i$ denotes the fiberwise double $E_i^l\cup_\id E_i^r$ with a left copy $E_i^l$ and a right copy $E_i^r$ glued together along their isomorphic boundaries and the induced local coefficient system $\F_i^l\cup_\id\F_i^r.$
\end{Axiom}

Now suppose again that $p:E\to B$ is a  bundle with closed fiber $F$ and local coefficient system $\F$ on $E.$ Let $q:D\to E$ be a $S^n$-bundle which is isomorphic to the sphere bundle of a vector bundle. We get the local coefficient system $q^*\F$ on $D$ by pulling back $\F$ along $q.$

\begin{Axiom}[geometric transfer]
In the situation above, for a twisted torsion invariant $\tau$, we have the following relation between the torsion class $\tau_B(D,q^*\F)\in H^{2k}(B;\R)$ of $D$ as a bundle over $B$  and the torsion class $\tau_E(D,q^*\F)\in H^{2k}(E;\R)$ of $D$ as a bundle over $E$:
	$$\tau_B(D,q^*\F)=\chi(S^n)\tau_B(E,\F)+\tr_B^E(\tau_E(D,q^*\F)),$$
where $\chi$ denotes the Euler class, $\tr_B^E:H^{2k}(E;\R)\to H^{2k}(B;\R)$ the trace, and $\tau_E(D,q^*\F)$ the twisted torsion class of $D$ over $E$.
\end{Axiom}

\begin{Remark} 
We have $\chi(S^n)=2$ or $0$ depending on whether $n$ is even or odd.
\end{Remark}

\begin{Remark}
If we take a twisted torsion class $\tau_{2k}$, we will get a non-twisted torsion class 
	$$\tau_{\textnormal{non-tw.}}(E):=\tau(E,\1)\in H^{4k}(B;\R),$$
where $E\to B$ is a bundle and $\1$ the constant local system on $E.$ We will denote this non-twisted torsion invariant simply by $\tau(E)$ without any local system in the argument.
\end{Remark}
Since there are no higher torsion invariants in degree $4l+2=2k,$ we also need the following Axiom:

\begin{Axiom}[triviality]
For a twisted torsion invariant in degree $4l+2,$ we have for every bundle $E\to B$ and the constant local system $\1$ on $E$
	$$\tau(E,\1)=0\in H^{4l+2}(B;\R).$$
\end{Axiom}

These axioms so far were only modifications of the axioms for non-twisted torsion invariants. We also need some axioms concerning the local system $\F$ on $E:$

\begin{Axiom}[additivity for coefficients]
If $\F=\bigoplus_i\F_i$ for local systems $\F_i$ on $E$ and a bundle $E\to B$, we have for every twisted torsion invariant $\tau$
	$$\tau(E,\F)=\sum_i\tau(E,\F_i).$$
\end{Axiom}

\begin{Axiom}[transfer/induction for coefficients]
If $\tilde E\to B$ and $E\to B$ are bundles and $\pi:\tilde E\to E$ is a finite fiberwise covering, then we have for every local system $\F$ on $\tilde E$
	$$\tau(\tilde E, \F)=\tau(E,\pi_*\F),$$
where $\pi_*$ denotes the push-down operator for local systems.

\end{Axiom}

\begin{Remark} K. Igusa proposed this axiom originally in the following form (see \cite{Igusa4}), which corresponds to our formulation:\\
If $G$ is a group that acts freely and fiberwise on $E\to B,$ $H$ is a subgroup of $G,$ and $V$ is a unitary representation of $H$, then the torsion of the orbit bundles $E/G,\,E/H\to B$ are related by
	$$\tau(E/G, Ind_H^GV)=\tau(E/H, V).$$
\end{Remark}

Besides these two, we also need a very specific continuity axiom stated in section 5.2, which we will only use once. Following a conjecture of Milnor, we should be able to drop this axiom. We will discuss this matter, as we get to it. We will also state the axiom in its specific form at the one point, where we will use it.

\section{Statement}

\subsection{Examples of Twisted Higher Torsion Invariants}

There are two prominent examples for higher torsion. The first one is the higher Franz Reidemeister torsion or Igusa-Klein torsion 
	$$\tau^{IK}_k(E,\partial_0 E,\F)\in H^{2k}(B;\R),$$
which is defined for any unipotent bundle pair $(F,\partial_0 F)\to (E,\partial_0 E)\to B$ with $\partial_0 E\subseteq \partial^v E$ and local system $\F$ on $E$ (for details, see \cite{Igusa2}).\\
K. Igusa proved the following result in \cite{Igusa3}:
\begin{Theorem}
Igusa-Klein torsion invariants are higher twisted torsion invariants for  bundles with closed fibers.
\end{Theorem}
Besides this torsion, we also have the Miller-Morita-Mumford classes in degree $4l$ with $l\in \N$
	$$M^{2l}(E):=\tr_B^E((2l!)ch_{4l}(T^vE)),$$
where $ch_{4l}(T^vE)=\frac 1 2 ch_{4l}(T^vE\otimes \C)$. We will consider this to be a real characteristic class. K. Igusa also showed that this class is a higher non-twisted torsion invariant (see \cite{Igusa1}). To make it a higher twisted torsion invariant we simply define for a $m-$dimensional local system $F$ on $E$
	$$M^{2l}(E,\F):=mM^{2l}(E)\in H_{4l}(B;\R).$$
Furthermore we set
	$$M^{2l+1}(E,\F):=0,$$
since there is no non-twisted torsion in degree $2k=2(2l+1),$ and the twisted MMM torsion always induces non-trivial non-twisted torsion. Knowing that the MMM class is a non-twisted torsion invariant as shown in \cite{Igusa1} (and therefore fulfills the first three axioms) it is now easy to see:
\begin{Theorem}
The twisted MMM class is a higher twisted torsion invariant.
\end{Theorem}
\begin{proof} Triviality is obvious, since the whole class is trivial in dimension $2k=2(2l+1).$ Additivity follows from the fact that the MMM class has a coefficient depending linearly on the dimension of the local system. We have not introduced the continuity axiom yet, but it will require the torsion class to depend continuously on the local system. It is met by the MMM class, since a transfer of coefficients leaves the MMM class constant.\\

So it only remains to show the transfer for coefficients axiom. Take two bundles $\tilde E\to B$ and $E\to B$ so that $\pi:\tilde E\to E$ is aN $n$-fold fiberwise covering. Assume we have a $m$-dimensional local system $\F$ on $\tilde E.$ Since the MMM class only takes the dimension of the local system in account we have
	$$M(E,\pi_* \F)=mnM(E)$$
and since $M(\tilde E,\F)=mM(\tilde E)$ it remains to show
	$$M(\tilde E)=nM(E).$$
Let us first recall the definition of the MMM class as 
	$$M^{4l}(E)=tr_B^E((2l!)ch_{4l}(T^vE)),$$
where $tr_B^E(x)=p_*(x\cup e(T^vE))$ with the push-down operator $p_*:H^{*+l}(E;\Z)\to H^*(B;\Z)$ where $l$ is the dimension of $F.$ Now we have the $n$-fold covering $\pi:\tilde E\to E$ and the following pull-back diagram:
	$$\xymatrix{
			T^v\tilde E\ar[r]\ar[d]		&	T^vE\ar[d]\\
			\tilde E\ar[r]^\pi				&	E.
		}
	$$
By naturality, this implies $ch_{4l}(T^v\tilde E)=\pi^*ch_{4l}(T^vE)$ and $e(T^v\tilde E)=\pi^*e(T^vE).$ Furthermore we have the following commutative diagram relating the push-down operators for $\tilde E$ and $E:$
	$$\xymatrix{
			H^{*+l}(\tilde E;\Z)\ar[dr]_{p_*}	&	&	H^{*+l}(E;\Z)\ar[ll]^{\pi^*}\ar[dl]^{n\cdot p_*}\\
			&	H^*(B;\Z).
		}
	$$
Putting everything together we calculate
	\begin{eqnarray*}
		M^{2l}(\tilde E)	&	=	&	p_*(ch_{4l}(T^v \tilde E)\cup e(T^v\tilde E))\\
											&	=	&	p^*(\pi^*(ch_{4l}(T^v(E))\cup e(T^vE)))\\
											&	=	&	np^*(ch_{4l}(T^vE)\cup e(T^vE))\\
											&	=	&	nM^{4l}(E)
	\end{eqnarray*}
and this completes the proof.
\end{proof}

Now we know that for any bundle $F\to E\to B$ with closed $l$-dimensional fiber $F,$ twice the transfer map $\tr_B^E$ is zero, if $l$ is odd. Therefore we get

\begin{Proposition}
$M^k(E,\F)=0$ for closed odd dimensional fiber $F$.
\end{Proposition}



\subsection{The Space of Twisted Torsion Invariants}

Now we are moving on to the space of higher twisted torsion invariants in degree $2k$. We begin with the following elementary observation:

\begin{Lemma}
For each $k$, the set of all twisted torsion invariants of degree $2k$ is a  vector space over $\R.$
\end{Lemma}
\begin{proof} The axioms are homogeneous linear equations in $\tau$.
\end{proof}

Of course, the same statement holds for the set of non-twisted higher torsion invariants. K. Igusa proved for the space of non-twisted higher torsion invariants in \cite{Igusa1}:

\begin{Theorem}
For any $k$ the space of higher non-twisted torsion invariants in degree $4k$ is two dimensional and spanned by the non-twisted MMM class $M^{4k}$ and the non-twisted Igusa-Klein torsion $\tau^{IK}_{4k}$. In other words, for any non-twisted torsion invariant $\tau$ there exist unique $a,b\in \R$ so that 
	$$\tau=a\tau^{IK}+bM.$$
\end{Theorem}

Now, let $Top_{fin}$ be the full subcategory of $Top$ of topological spaces with finite fundamental group and $Top_{sim}$ the full subcategory of simple topological spaces. A space $F$ is called simple if the fundamental group $\pi_1 F$ acts trivially on the homology $H_*(F;\Z).$ If we restrict a twisted torsion invariant to bundles with fibers in $Top_{sim}$ and base in $Top_{fin}$ we get the main theorem:



\begin{Theorem}[Main Theorem] 
In the setting above, the space of higher twisted torsion invariants in degree $2k$ on bundles with simple fibers and rationally simply conntected base is two dimensional and spanned by the twisted MMM class and the twisted Igusa-Klein torsion, if $k$ is even, and one dimensional and spanned by the Igusa-Klein torsion, if $k$ is odd. In other words, for any twisted torsion invariant $\tau$, there exists a unique $a\in \R$ and there exists a (not necessarily unique) $b\in \R$, so that
	$$\tau=a\tau^{IK}+bM.$$
\end{Theorem}

\begin{Remark}
If $k$ is even, we get a non-twisted torsion invariant from the twisted one by always inserting the trivial representation. Then the numbers $a$ and $b$ used in both theorems above will be the same.
\end{Remark}

The rest of the paper is dedicated to the prove of the main theorem.

\subsection{The Scalars $a$ and $b$}

Now we try to determine the scalars $a$ and $b,$ by evaluation on $S^1$-bundles.

\subsubsection{In degree $2k=4l$}

Let us first look at a twisted torsion invariant in degree $2k=4l.$ In this case the scalars must be the same as the ones we get for the corresponding non-twisted torsion. To determine them we follow Igusa's approach\cite{Igusa1} and look at the universal $S^1\cong U(1)\cong SO(2)$-bundle $\lambda$ over $\C\P^\infty=BU(1)$. Furthermore, let $S^1(\lambda)$ be the associated circle bundle with $\lambda$ and $S^2(\lambda)$ the $S^2$-bundle associated with $S^1(\lambda)$ (by fiberwise suspension of $S^1(\lambda)$). Since the cohomology ring of $\C\P^\infty$ is a polynomial algebra generated by $c_1(\lambda)$, the cohomology group $H^{2k}(\C\P^\infty;\R)\cong\R$ is generated by $ch_{2k}(\lambda)=c_1^k/k!$.\\
From this, we immediately get scalars $s_1,s_2\in\R$ for any twisted torsion invariant in degree $2k=4l$ with
	 $$\tau(S^1(\lambda))=s_1 ch_{2k}(\lambda)$$
and
	$$\tau(S^2(\lambda))=s_2 ch_{2k}(\lambda).$$
Furthermore we have the following two propositions (from \cite{Igusa3} and \cite{Igusa1}):

\begin{Proposition}
Substituting $2l=k$ we get
	$$\tau_{2l}^{IK}(S^n(\lambda))=(-1)^{l+n}\zeta(2l+1)ch_{4l}(\lambda).$$
\end{Proposition}

\begin{Proposition}
$M_k(S^2(\lambda))=2k!ch_{2k}(\lambda)!.$
\end{Proposition}

Now we are taking in account that the MMM class is trivial on odd dimensional fibers, and therefore we get that $\tau(S^1(\lambda))=a\tau^{IK}(S^1(\lambda)).$ From this we get 
	$$a=s_1/((-1)^{1+l}\zeta(2l+1)).$$
Looking at the $S^2(\lambda)$ case, we have
	$$s_2=a(-1)^l\zeta(2l+1)+b2k!=-s_1+b2k!$$
and therefore
	$$b=\frac{s_1+s_2}{2k!}.$$

\subsubsection{In degree $2k=4l+2$}

Now let the degree be $2k=4l+2.$ In this case $\tau$ does not define a non-trivial non-twisted torsion invariant. On the other hand we also just need to determine $a$ since the MMM class vanishes in this degree.\\
In the degree $2k=4l+2$ case we cannot use the standard universal bundle for linear $S^1$-bundles $ES^1\to BS^1,$ since $ES^1$ is contractible and therefore will not admit a non-constant local system. But we can replace it by a very similar construction. First, recall that $ES^1$ can be constructed as follows: Take $S^1\subseteq \C$ and $S^{2N-1}\subseteq \C^N$. Then we have a fibration $S^1\into S^{2N-1}\to \C\P^{N-1}$. Taking the limit of this will yield an $S^1$-bundle with total space $S^\infty$, which is contractible and therefore the universal $S^1$-principal bundle $S^1\into S^{\infty}\to \C\P^\infty$.\\
Now we can look at a $\Z/n$-action on $S^1$ given by multiplication with the primitive $n$-th root of unity $e^{2\pi i/n}$. This will give rise to a fiberwise $\Z/n$-action on the bundle $S^1\into S^{2N-1}\to \C\P^{N-1}$. The action of $\Z/n$ on $S^{2N-1}$ is by construction the same as the one being taken to get a lens space $L_n^{2N-1}$ as quotient out of $S^{2N-1}$. Therefore taking the fiberwise quotient under the given $\Z/n$-action gives a bundle (since $S^1/n\cong S^1$)
	$$S^1\into L_n^{2N-1}\to \C\P^{N-1},$$
which yields in the limit to
	$$S^1\into L_n^\infty\to \C\P^\infty.$$
We will refer to this bundle as $S^1(\lambda)/n$, since it has the $S^1$-bundle associated with the universal line bundle as its $n$-fold covering. The $n$-fold Galois covering $S^{2N-1}\to L_n^{2N-1}$ gives a covering $S^{2N-1}\times \C\to L_n^{2N-1}$ where a fixed generator of $\Z/n$ acts on $\C$ by multiplication with an $n$-th root of unity $\zeta_n.$ Using this we can make the following important definition.

\begin{Definition} In the setting above, the non-constant local system $\F_{\zeta_n}$ on $L^{2N-1}_n$ is defined to be the non-constant local system of the sections of the covering $S^{2N-1}\times \C\to L_n^{2N-1}.$ The non-constant local system $\F_{\zeta_n}$ on $L^{\infty}_n$ is defined as the limit of these local systems on $L_n^{2N-1}.$
\end{Definition}

Again, we can use the fact that the cohomology of $\C\P^\infty$ is a group ring and that $H^{2k}(\C\P^\infty;\R)$ will be spanned by $ch_{2k}(\lambda)$ and therefore 
	$$\tau(S^1(\lambda)/n,\F_{\zeta_n})=s_1ch_{2k}(\lambda).$$
Furthermore we have again the following result from K. Igusa \cite{Igusa3}:

\begin{Proposition}
For the Igusa-Klein torsion we have
	$$\tau^{IK}(S^1(\lambda)/n,\F_{\zeta_n})=-n^kL_{k+1}(\zeta_n)ch_{2k}(\lambda),$$
where $L_{k+1}$ denotes the polylogarithm
	$$L_{k+1}(\zeta):=\Re\left(\frac 1{i^k}\sum_{l=1}^\infty\frac{\zeta_n^l}{n^{k+1}}\right).$$
\end{Proposition}

Putting this together we get
	$$a=-s_1/(n^kL_{k+1}(\zeta)).$$
We will prove later that $a$ is independent of the choice of the local system.

\section{Extension of Higher Twisted Torsion}

In this section, which strictly follows  the corresponding section in \cite{Igusa1}, we will extend a twisted torsion invariant $\tau$ to bundles whose fibers have a boundary and then define a relative torsion for  bundle pairs.

\begin{Definition}
A pair of bundles $(F,\partial_0)\to (E,\partial_0)\to B$ is called a bundle pair, if the vertical boundary $\partial^vE$ is the union of two subbundles $\partial^vE=\partial_0E\cup \partial_1 E$, which meet along their common boundary $\partial_0E\cap\partial_1 E=\partial^v\partial_0E=\partial^v\partial_1 E.$ 
\end{Definition}

\begin{figure}[h]
\label{bundle pair}
\begin{center}
\begin{picture}(150,130)
\setlength{\unitlength}{1cm}
\put(0,0){\includegraphics[width=2cm]{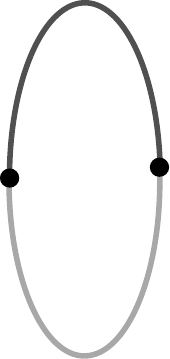}}
\put(1.4,0.0){$\partial_0F_x\cong I$}
\put(1.6,4){$\partial_1F_x\cong I$}
\put(0.4,2.5){\line(4,1){2.2}}
\put(2.0,2.6){\line(2,1){1}}
\put(2.6,3.2){$\partial\partial_iF_x\cong\{0,1\}$}
\end{picture}
\end{center}
\caption{The fiber over $x$ of a bundle pair with fiber $F\cong D^2$}
\end{figure}	
\subsection{Bundles with Vertical Boundary}

First we define the higher twisted torsion on bundles with vertical boundary:

\begin{Definition}[Higher twisted torsion for bundles with vertical boundary] Suppose $F\into E\to B$ is a  bundle with vertical boundary $\partial^v E\to B$ and local coefficient system $\F$ on $E$ and $\tau$ is a higher twisted torsion invariant. Then the twisted torsion of the bundle with boundary is defined by
	$$\tau(E,\F):=\frac 1 2(\tau (DE,\F^l\cup_{\id}\F^r)+\tau(\partial^v E,\F_{|\partial^vE})),$$
where $DE:=E^l\cup_{\id} E^r$ denotes the fiberwise double as before.
\end{Definition}
 At first we have the following Lemma:

\begin{Lemma} Suppose $E_i$ ($1\leq  i\leq 4$) are  bundles over $B$ with local systems $\F$ so that there are isomorphisms $\phi_{ij}:\partial^v E_i\to \partial^v E_j$ for all $i<j$ and the local systems satisfy $(\F_i)_{|\partial^v E_i}\cong\phi_{ij}^* (\F_j)_{|\partial^v E_j}$. Then
	$$\tau(E_1\cup_{\phi_{12}} E_2,\F_1\cup_{\phi_{12}}\F_2)
	+\tau(E_3\cup_{\phi_{34}} E_4,\F_3\cup_{\phi_{34}}\F_4)$$$$
	=\tau(E_1\cup_{\phi_{13}} E_3,\F_1\cup_{\phi_{13}}\F_3)
	+\tau(E_2\cup_{\phi_{24}} E_4, \F_2\cup_{\phi_{24}}\F_4).$$
\end{Lemma}
\begin{proof} Both sides are equal to $\frac 1 2 \sum_i \tau(DE_i,\F_i^l\cup_{\id}\F_i^r)$ by the additivity axiom with the usual left and right copies of $E_i$ and $\F_i.$
\end{proof}

\begin{Lemma}
Let $E$ be a  bundle with local coefficient system $\F$ on $E$, then 
	$$\tau(\partial^v E, \F_{|\partial^vE})=\tau(\partial^v(E\times D^2),\F_{ind}),$$
where $\F_{ind}$ is the induced representation explained below.
\end{Lemma}
\begin{proof} We have $\partial^v(E\times D^2)=\partial^v E\times D^2\cup_{\partial^vE\times S^1} E\times S^1$. Furthermore we have a local systems $\F_1$ on $\partial^v E\times D^2$ and $\F_2$ on $E\times S^1$ given by the pull-back along the projection to the $E$-factor of the local system there. These two local systems will be canonically isomorphic on $\partial^vE \times S^1$ and therefore can be glued together to the local system $\F_{ind}$ on $\partial^v(E\times D^2).$ Additivity guaranties that
	$$\tau(\partial^v(E\times D^2),\F_{ind})=\frac 1 2(\tau(\partial^v E\times S^2,\F_1^l\cup_{\id}\F_1^r)+\tau(DE\times S^1,\F_2^l\cup_{\id}\F_2^r)),$$
with the usual induced representations on the fiberwise doubles. But we have $\tau(\partial^vE\times S^2,\F_1^l\cup_{\id} F_1^r)=2\tau(\partial^vE, \F_{|\partial^v E})$ and $\tau(DE\times S^1, \F_2^l\cup_{\id}\F_2^r)=0$ by the transfer axiom.
\end{proof}

\begin{Lemma}[additivity in the boundary case]\label{additivity}
Suppose $E$ is a bundle over $B$ and $(E_1,\partial_0)$ and $(E_2,\partial_0)$ are  bundle pairs such that $E_1, E_2\subseteq E,$ $\partial_0 E_1=\partial_0 E_2=E_1\cap E_2$ and $E=E_1\cup E_2.$ Let $\F$ be a local system on $E$ and $\F_1:=\F_{|E_1}$ and $\F_2:=\F_{|E_2},$ then
	$$\tau(E_1\cup E_2, \F)=\tau(E_1,\F_1)+\tau(E_2,\F_2)-\tau(E_1\cap E_2,\F_{|E_1\cap E_2}).$$
\end{Lemma}
\begin{proof} We extend the terms, using the defining equations
	\begin{eqnarray*} 
		\tau(E_i,\F_i)					&	:=	&	\frac 1 2 (\tau(DE_i,\F_i\cup_\id\F_i)+\tau(\partial^v E_i,(\F_i)_{|\partial^vE_i}))\\
		\tau(E_1\cup E_2,\F)	&	:=	&	\frac 1 2 (\tau(\partial^v(E_1\cup E_2),\F_{|\partial^v(E_1\cup E_2)})\\
													&	+		&	\tau(D(E_1\cup E_2),\F\cup_\id\F))\\
		\tau(E_1\cap E_2,\F{|E_1\cap E_2})	&	:=	&	\frac 1 2 (\tau(D(E_1\cap E_2),\F_{|E_1\cap E_2}\cup_\id\F_{E_1\cap E_2})\\
													&	+		&	\tau(\partial^v(E_1\cap E_2),\F_{|\partial^v(E_1\cap E_2)})).
	\end{eqnarray*}
Recall that $\partial^vE_i=\partial_0 E_i\cup \partial_1 E_i$, where they meet along $\partial_0 E_i\cap \partial_1 E_i=\partial^v\partial_0E_i=\partial^v\partial_1 E_i$. An application of an earlier Lemma gives 
	\begin{eqnarray*} 
		\tau(\partial^v E_1,\F_{|\partial^v E_1})+\tau(\partial^v E_2,\F_{|\partial^v E_2})	
										&	=	&	\tau(\partial_0 E_1\cup \partial_1E_1,\F_{|\partial_0 E_1}\cup_{\partial^v\partial_0E_1}\F_{|\partial_1E_1})\\
										&	+	&	\tau(\partial_0 E_2\cup \partial_1 E_2,\F_{|\partial_0 E_2}\cup_{\partial^v\partial_0E_2}\F_{|\partial_1E_2})\\
										&	=	&	\tau(\partial_1 E_1\cup \partial_1 E_2,\F_{|\partial_1 E_1}\cup_{\partial^v\partial_0E_1}\F_{|\partial_1E_2})\\
										&	+	&	\tau(\partial_0 E_1\cup \partial_0 E_2,\F_{|\partial_0 E_1}\cup_{\partial^v\partial_0E_1}\F_{|\partial_0E_2})\\
										&	=	&	\tau(\partial^v(E_1\cup E_2),\F_{|\partial^v(E_1\cup E_2)})\\
										&	+	&	\tau(D(E_1\cap E_2),\F_{|E_1\cap E_2}^l\cup_\id\F_{|E_1\cap E_2}^r).
	\end{eqnarray*}
The same Lemma gives furthermore (where the shared boundary is $D(\partial_0E_1)=D(\partial_0 E_2)$)
	\begin{eqnarray*}
		\tau(DE_1,\F)+\tau(DE_2,\F)
										&	=	&	\tau((E_1\cup_{\partial_1E_1}E_1)\cup I\times(E_1\cap E_2),\F)\\
										&	+	&	\tau((E_2\cup_{\partial_1E_2}E_2)\cup I\times(E_1\cap E_2),\F)\\
										&	=	&	\tau((E_1\cup_{\partial_1E_1}E_1)\cup (E_2\cup_{\partial_1 E_2}E_2),\F)\\
										&	+	&	\tau(D(I\times(E_1\cap E_2)),\F)\\
										&	=	&	\tau(D(E_1\cup E_2),\F)+\tau(D(I\times(E_1\cap E_2),\F)).
	\end{eqnarray*}
To unburden notation, we denoted every local system by $\F.$ They are all induced naturally by restriction or gluing along the identity from the given local system $\F$ on $E_1\cup E_2.$ Especially the local system $\F_3$ on $I\times (E_1\cap E_2)$ is constant on $I.$ Furthermore, we have
$$\tau(D(I\times(E_1\cap E_2)),\F_3\cup_\id \F_3)=\tau(\partial^v(D^2\times (E_1\cap E_2)),\F_{ind})=\tau(\partial^v(E_1\cap E_2),\F_{|\partial^v(E_1\cap E_2)}),$$
where $\F_{ind}$ is the induced local system from the previous lemma. Putting all this together gives the Lemma.
\end{proof}

To develop the transfer formula in the boundary case, we first need the following Lemma:

\begin{Lemma}
	If $E=E_1\cup E_2$ is a union of two smooth bundles along their common vertical boundary $\partial^vE_1=\partial^vE_2=E_1\cap E_2,$ we have
		$$\tr_B^E(x)=\tr_B^{E_1}(x|E_1)+\tr_B^{E_2}(x|E_2)-\tr_B^{\partial^v E_1}(x|\partial^v E_1).$$
\end{Lemma}

\begin{Lemma}[Transfer in the boundary case] 
Let $X\to D\stackrel{q}{\to} E$ be an oriented disk or sphere bundle over a bundle $F\to E\to B$ with local coefficient system $\F$ on $E$. As for the transfer axiom this pulls up to a local coefficient system $q^*\F$ on $D$ and we get
	$$\tau_B(D,q^*\F)=\chi(X)\tau(E,\F)+\tr_B^E(\tau_E(D),q^*\F).$$
\end{Lemma}
\begin{proof} We first consider the case where $F$ does not have a boundary and is therefore closed. The case where $X$ is closed, and so a sphere bundle, is simply the transfer axiom. So we are just considering the case $X=D^n$ for the $n$-dimensional linear disk bundle $D=D(\xi)$ and $F$ closed. We have the following two examples of the original transfer axiom (using $S^n(\xi)=DD(\xi)$)
	\begin{eqnarray*}
		\tau_B(S^n(\xi),(q^*\F)^l\cup_\id (q^*\F)^r)&	=	&	\chi(S^n)\tau(E,\F)+\tr_B^E(\tau_E(S^n(\xi)),(q^*\F)^l\cup_\id (q^*\F)^r)\\
		\tau_B(S^{n-1}(\xi),(q^*F)_{|S^{n-1}(\xi)})	&	=	&	\chi(S^{n-1})\tau(E,\F)+\tr_B^E(\tau_E(S^{n-1}(\xi)),(q^*\F)_{|S^{n-1}(\xi)})
	\end{eqnarray*}
Taking half of the vertical sums here we get
	\begin{eqnarray*}
		\tau_B(D(\xi),q^*D(\xi))	&	=	&	\tau(E,\F)+\tr_B^E(\frac 1 2(\tau_E(S^n(\xi),(q^*\F)^l\cup_\id (q^*\F)^r)\\
															&	+	&	\tau_E(S^{n-1},(q^*\F)_{S^{n-1}(\xi)})))\\
										&	=	&	\chi(D(\xi))\tau(E,\F)+\tr_B^E(\tau_E(D(\xi)),q^*\F),
	\end{eqnarray*}
as desired.\\
Now we are turning to the case, where $F$ has a boundary and $X$ is either a sphere or disk bundle. Write $DE=E^l\cup E^r$ for a left and right copy $E^l$ and $E^r$ of $E$ that meet along their vertical boundary. Name the disk bundles over $E^l$ and $E^r$ by $D^l$ and $D^r$ meeting at $D^l\cap D^r=q^{-1}(\partial^v E)$, where $q:D\to E$ denotes the bundle map. Now we can use the following transfer formulas (since they resemble the case, where $F$ is close)
	\begin{eqnarray*}	
		\tau_B(D^l\cup D^r,(q^*\F)^l\cup_\id (q^*\F)^r)	&	=	&	\chi(X)\tau(E^l\cup E^r,\F^l\cup_\id\F^r)\\
																					&	+	&	\tr_B^{E^l\cup E^r}(\tau_{E^l\cup E^r}(D^l\cup D^r,(q^*\F)^l\cup_\id (q^*\F)^r))\\
		\tau_B(D^l\cap D^r,q^*F_{|D^l\cap D^r})	&	=	&	\chi(X)\tau(E^l\cap E^r,\F_{|E^l\cap E^r})\\
																			&	+	&	\tr_B^{E^l\cap E^r}(\tau_{E^l\cap E^r}(D^l\cap D^r,q^*\F_{|D^l\cap D^r})).
	\end{eqnarray*}
Next we can use the additivity in the boundary case and the trace formula and by taking half of the vertical sum, we get
	$$\tau_B(D,q^*\F)=\chi(X)\tau(E,\F)+\tr_B^E(\tau_E(D),q^*\F).$$
\end{proof}

\subsection{Relative Torsion}

Now we turn again to  bundle pairs.

\begin{Definition}[relative Torsion]
For a  bundle pair $(F,\partial_0)\to (E,\partial_0)\to B$ with local coefficient system $\F$ on $E$ we define the relative torsion to be
	$$\tau(E,\partial_0, \F):=\tau(E,\F)-\tau(\partial_0 E,\F_{|\partial_0 E}).$$
\end{Definition}

We get the following Propositions, which are easy to prove:

\begin{Proposition}[relative additivity] Suppose $E\to B$ is a smooth bundle with local coefficient system $\F,$ which can be written as the union of two subbundles $E=E_1\cup E_2,$ which meet along a subbundle of their respective vertical boundaries $E_1\cap E_2=\partial_0 E_2\subseteq \partial^v E_1$. Let $\partial^v E_1=\partial_0 E_1\cup\partial_1E_1$ be a decomposition of $\partial^v E_1$, so that $\partial_0E_2\subseteq \partial_1 E_1$ and $(E_i,\partial_0)\to B$, $i=1,2$ are  smooth bundle pairs. Then $(E,\partial_0 E)\to B$ is  and
	$$\tau(E_1\cup E_2,\partial_0 E_1,\F)=\tau(E_1,\partial_0,\F_{|E_1})+\tau(E_2,\partial_0,\F_{|E_2}).$$
\end{Proposition}
\begin{proof} By Proposition \ref{additivity} both sides of the equation are equal to $$\tau(E_1,\F_{|E_1})+\tau(E_2,\F_{|E_2})-\tau(E_1\cap E_2,\F_{|E_1\cap E_2})-\tau(\partial_0 E_1,\F_{|\partial_0E_1})$$.
\end{proof}

We also have the following proposition without proof.

\begin{Proposition}
Suppose $(E,\partial_0)\to B$ is a union of two  bundle pairs $(E_i,\partial_0E_i)$ with local coefficient system $\F$. That means we have $E=E_1\cup E_2$ and $\partial_0 E=\partial_0 E_1\cup \partial_0 E_2$ with $E_1\cap E_2\subseteq \partial_1 E_1\cap \partial_1 E_2$. Let $E_0=E_1\cap E_2$ and $\partial_0E_0=E_0\cap\partial_0E$ and suppose $(E_0,\partial_0)$ is a  bundle pair. Then $(E,\partial_0)$ is  and 
	$$\tau(E,\partial_0,\F)=\tau(E_1,\partial_0,\F_{|E_1})+\tau(E_2,\partial_0,\F_{|E_2})-\tau(E_0,\partial_0,\F_{|E_0}).$$
\end{Proposition}

To state the transfer axiom in the relative case, we need the relative transfer:
	$$\tr_B^{(E,\partial_0)}:H^*(E;\R)\to H^*(B;\R),$$
given by
	$$\tr_B^{(E,\partial_0)}=p_*(x\cup e(E,\partial_0)),$$
where 
	$$p_*:H^{*+k}(E,\partial^v E;\R)\to H^*(B;\R)$$
is the push-down operator and 
	$$e(E,\partial_0)\in H^k(E,\partial^vE;\R)$$
is the relative Euler class given by the pull-back of the Thom class of the vertical tangent bundle $T^vE$ along any vertical tangent vector field which is nonzero along the vertical boundary $\partial^v E$ and which points inwards along $\partial_0 E$ and outward along $\partial_1 E$. The relative transfer satisfies the following two equations for any $x\in H^*(E;\R)$:
	\begin{eqnarray*}
		\tr_B^{(E,\partial_0)}(x)	&	=	&	\tr_B^E(x)-\tr_B^{\partial_0E}(x|\partial_0 E)\\
		\tr_B^{(E,\partial_1)}		&	=	&	(-1)^k\tr_B^{(E,\partial_0)}.
	\end{eqnarray*}
And $\tr_B^{(E,\partial_0)}\circ p^*:H^*(B)\to H^*(B)$ is multiplication by the relative Euler characteristic of the fiber pair $(F,\partial_0)$:
	$$\chi(F,\partial_0):=\chi(F)-\chi(\partial_0F).$$

\begin{Proposition}[relative transfer]
Let $(F,\partial_0)\to(E,\partial_0)\to B$ and $(X,\partial_0)\to (D,\partial_0)\stackrel q \to E$ be  bundle pairs with local system $\F$ on $E$, so that the second bundle is an oriented linear $S^n$ or $D^n$ bundle with $\partial_0 X=S^{n-1}, D^{n-1}$ or $\emptyset$. Then
	$$\tau_B(D,\partial_0 D\cup q^{-1}\partial_0 E,q^*\F)=\chi(X,\partial_0)\tau(E,\partial_0,\F)
				+\tr_B^{E,\partial_0}(\tau_E(D,\partial_0,q^*\F)).$$
\end{Proposition}

\begin{proof}
We already proved the case where both $\partial_0F$ and $\partial_0 X$ are empty. The case where $\partial_0 X$ is empty follows easily from
	$$\tau_B(D,q^{-1}\partial_0E,q^*\F)=\tau_B(D,q^*\F)-\tau_B(q^{-1}\partial_0 E,q^*\F_{|q^{-1}\partial_0E}):$$
The first term on the right hand side is 
	$$\tau_B(D,q^*\F)=\chi(X)\tau(E,\F)+\tr_B^E(\tau_E(D,q^*\F))$$
and the second term gives 
	$$\tau_B(q^{-1}\partial_0 E,q^*\F_{|q^{-1}\partial_0 E})=\chi(X)\tau_B(\partial_0E,\F_{|\partial_0E})+
														\tr_B^{\partial_0E}(\tau_{\partial_0 E}(q^{-1}\partial_0E,q^*\F_{|q^{-1}\partial_0E})).$$
Using $\chi(X,\partial_0 X)=\chi(X)$ (since $\partial_0 X=\emptyset$) and the formula $\tr_B^{E,\partial_0}(x)=\tr_B^E(x)-\tr_B^{\partial_0 E}(x|\partial_0 E)$ gives the Lemma in that case.\\
The general case follows from the following two examples of the $\partial_0 X=\emptyset$ case (Again, we are unburdening notation by writing $\F$ for any local system induced naturally by pulling back along $q$ and restricting):
	\begin{eqnarray*}
		\tau_B(\partial_0D,\partial_0D\cap q^{-1}\partial_0E,\F)	&	=	&	\chi(\partial_0X)\tau(E,\partial_0,\F)
																							+\tr_B^{(E,\partial_0)}(\tau_E(\partial_0 D,\F))\\
		\tau_B(D,q^{-1}\partial_0E,\F)	&	=	&	\chi(X)\tau(E,\partial_0,\F)+\tr_B^{(E,\partial_0)}(\tau_E(D,\F))
	\end{eqnarray*}
Taking the second formula minus the first gives on the left hand side using the general additivity
	\begin{eqnarray*}
		\tau_B(D,q^{-1}\partial_0 E,\F)-\tau_B(\partial_0 D,\partial_0D\cap q^{-1}\partial_0 E,\F)
										&	=	&	\tau_B(D,\F)+\tau_B(\partial_0D\cap q^{-1}\partial_0E,\F)\\
										&	-	&	\tau_B(q^{-1}\partial_0 E,\F)-\tau_B(\partial_0 D,\F)\\
										&	=	&	\tau_B(D,\F)-\tau_B(\partial_0 D\cup q^{-1}\partial_0 E,\F)\\
										&	=	&	\tau_B(D,\partial_0 D\cup q^{-1}\partial_0 E,\F)
	\end{eqnarray*}
and on the right hand side the desired term to prove the Lemma.
\end{proof}

\begin{Remark}
Note that we do not have an analogous to the product formula (Corollary 5.10 in \cite{Igusa1}).
\end{Remark}

We still have this important corollary from the transfer formula:

\begin{Corollary}[stability theorem]

If $(E,\partial_0)\to B$ is a  smooth bundle pair with local system $\F$ on $E$, then so is $(E\times D^n,\partial_0E\times D^n)$ and the relative torsion is the same:
	$$\tau(E\times D^n,\partial_0E\times D^n,\F\times \1)=\tau(E,\partial_0,\F),$$
where $\F\times\1$ is the local system constant on $D^n. $
\end{Corollary}

\section{Calculations for Higher Twisted Torsion}

\subsection{Reduction of the Representation}

In the following we will simplify the local system of the bundles.\\
Let $F\into E\to B$ be a fiber bundle and $\F$ a finite local system on $E.$ This corresponds to its holonomy cover $\tilde E\to E$ with finite transition group $G$ and representation $\rho:G\to U(m).$ On the other hand every finite covering $\tilde E\stackrel{G}{\to}E$ with representation $\rho:G\to U(m)$ gives us a local system $\F_{\rho}$ as the sections of the bundle $\tilde E\times_G\C^m\to E$ where $G$ acts on $\C^m$ via $\rho.$ This construction is a 1-1-correspondence.\\
Now let $H\subseteq G$ be a subgroup. From the covering $\tilde E\stackrel G \to E$ we get coverings $\pi_H:\tilde E/H\to E$ and $\tilde E\stackrel H \to \tilde E/H.$ Suppose we have a representation $\rho_H:H\to U(m)$ and thereby get a local system $\F_{\rho_H}$ on $\tilde E/H.$ Then we can either form the induced representation $Ind_H^G(\rho_H):G\to U(m)$ and its corresponding local system $\F_{Ind_H^G(\rho_H)}$ on $E$ or the local system $\pi_*\F_{\rho_H}$ on $E$ given by the push-down of the local system $\F_\rho.$ It follows from an easy calculation that
	$$\F_{Ind_H^G(\rho_H)}=\pi_*\F_{\rho_H}.$$
Let $\F$ be again a local system on $E$ corresponding to a finite covering $\tilde E\stackrel G \to E$ with representation $\rho:G\to U(m).$ Let $\mathcal H=\{H_i\}$ be the finite set of cyclic subgroups $H_i$ of $G.$ By Artin's induction theorem, we can write the character of $\rho$ rationally as linear combination of characters of one dimensional representations. Since we are working over $\C,$ we therefore can write $\rho$ rationally as a linear combination of one dimensional representations $\lambda_i:H_i\to U(1)$ and inductions thereof. Concretely we have
	$$n\rho\cong\bigoplus_in_iInd_{H_i}^G(\lambda_i)$$
with $n,n_i\in \Z.$\\
Let $\tau$ be a twisted torsion invariant and $\pi_i:\tilde E/H\to E$ be a covering, then we have using the transfer of coefficient axiom and the calculation above
	$$n\tau(E,\F)=\sum_in_i\tau(E,\F_{Ind_{H_i}^G(\lambda_i)})=\sum_in_i\tau(E,\pi_*\F_{\lambda_i})=\sum_i n_i\tau(\tilde E/H_i,\F_{\lambda_i})\in H^{2k}(B;\R).$$
Therefore it suffices for the rest of the paper to work with local systems with $n$-fold holonomy covers with cyclic transition group $\Z/n.$\\
\\
Now let $F\to E\to B$ be a bundle with local system $\F$ on $E$ and the base $B$ having a finite fundamental group. We have the universal covering $q:\tilde B\to B$ and pulling back $E$ along $q$ gives a bundle $\tilde E:=q^*E\to \tilde B$ with local system $\tilde \F:=q^*\F.$ Naturality implies 
	$$\tau(\tilde E,\tilde \F)=q^*\tau(E,\F)\in H^{2k}(\tilde B;\R).$$
Furthermore we know that $q^*:H^{2k}(B;\R)\to H^{2k}(\tilde B;\R)$ is injective. By this construction it suffices to prove the main theorem only on bundles with simply connected base.\\
For such a bundle $F\into E\to B$ the induced local system $\F_x$ on the fiber $F_x$ over $x\in B$ does not depend on $x$ and completely determines the local system $\F$ on $E.$ Therefore we will always look at bundles $F\into E\to B$ with simply connected base $B$ and local system $\F$ on $F$ (and thereby $E$).

\subsection{Twisted Torsion for $S^1$-Bundles}

Our goal in this section is to calculate all higher torsion invariants on $S^1$-bundles. To be precise, we want to show the following theorem:
\begin{Theorem}
For every $S^1$-bundle $S^1\into E\to B$ with $B$ simply connected and local system $\F$ on $E$ with $n$-fold holonomy cover $\tilde E_n\to E$ every twisted torsion invariant $\tau$ is given by
	$$\tau(E,\F)=a\tau^{IK}(E,\F),$$
where $a$ is the scalar defined earlier.
\end{Theorem}

We will follow an approach K. Igusa introduced in \cite{Igusa4}. \\
Since $BDiff(S^1)\simeq BSO(2)$ it suffices to look at linear $S^1$-bundles. These pull back from the universal $S^1$-bundle $S^1(\lambda)$ given by $S^1\into S^\infty\to \C\P^\infty$.\\
Let $E\to B$ be an $S^1$-bundle with local system $\F$ on $E$ inducing a finite holonomy covering. At first we look at the $n$-fold holonomy Galois covering
	$$\xymatrix
		{
			S^1\ar[d]\ar[r]^n					&	S^1\ar[d]\\
			\tilde E_n\ar[d]\ar[r]^n	&	E\ar[d]\\
			B\ar@{=}[r]								&	B.
		}
	$$
Now $\tilde E_n$ is again a linear $S^1$ bundle with fiberwise $n$-action. This will pull back equivariantly from the universal $S^1$-bundle $S^1(\lambda)$ given by $S^{\infty}\to \C\P^\infty,$ which also admits an $n$-action. Therefore $E$ will pull back from the quotient $S^1(\lambda)/(\Z/n).$ Also the local system $\F$ on $E$ will pull back from the local system $\F_{\zeta_n}$ on $S^{\infty}$ for some $n$-th root of unity $\zeta_n.$ We defined this earlier to be given by its holonomy cover $S^1(\lambda)\times \C\to S^1(\lambda)/n$ where the action on $\C$ is given by multiplication by $\zeta_n.$  So because of naturality it is enough to show 
\begin{Theorem}\label{basecase}
	$$\tau(S^1(\lambda)/n,\F_{\zeta_n})=a\tau^{IK}(S^1(\lambda)/n,\F_{\zeta_n})\in H^{2k}(\C\P^\infty; \R)$$
for all $n$ and $\zeta_n$.
\end{Theorem}
First we will prove two important Lemmas already introduced in \cite{Igusa4}. These Lemmas will isolate certain properties of $\tau(S^1(\lambda)/n,\F_{\zeta})$ thought of as a function of $\zeta.$ Then we can use a theorem of Milnor to show that the space of functions satisfying these properties is one dimensional and this will prove the theorem.

\begin{Lemma}

Suppose we have a bundle $E\to B$ and a free fiberwise $nm$-action $n,m\in\N$ on $E$. Then we have for any twisted torsion invariant and $n$-th root of unity $\zeta_n$
	$$\tau(E/n,\F_{\zeta_n^m})=\sum_{\xi^m=1}\tau(E/(nm),\F_{\xi\zeta_n}),$$
where the local systems $\F_{\zeta_n}$ on $E/n$ are given by the construction above.
\end{Lemma}

\begin{proof} Denote the projection $\pi:E/n\to E/(nm).$ The Peter Weyl theorem gives us
	$$\pi_*\F_{\zeta_n^m}=\bigoplus _{\xi^m=1}\F_{\xi\zeta_n}.$$
Now we can use the transfer of coefficients and the additivity axiom to get
	$$\tau(E/n,\F_{\zeta_n^m})=\tau(E/(nm),\pi_*\F_{\zeta_n^m})=\sum_{\xi^m=1}\tau(E/(nm),\F_{\xi\zeta_n}).$$
\end{proof}

\begin{Lemma}
For every linear $S^1$- bundle $E\to B$ and any $n$-th root of unity $\zeta_n$, we have for every twisted torsion class in degree $2k$
	$$\tau(E/(nm),\F_{\zeta_n})=m^k\tau(E/n,\F_{\zeta_n}).$$
\end{Lemma}
\begin{proof} Again we look at the universal circle bundle $S^1(\lambda),$ and by the naturality axiom it is enough to show the Lemma only on $E=S^1(\lambda)$. $S^1(\lambda)/m$ is again a circle bundle over $\C\P^\infty$ and therefore classified by a map
	$$f_m:\C\P^\infty\to \C\P^\infty.$$
In degree $2$ we can see (by looking at circle bundles over spheres $S^2$) that this map is multiplication by $m$ on $H^2$. Then it follows that $f^*_m$ is multiplication by $m^k$ on $H^{2k}(\C\P^\infty;\R)$. The classifying maps for $S^1(\lambda)/nm$ and $S^1(\lambda)/n$ are related by
	$$f_{mn}=f_n\circ f_m.$$
The Lemma now follows from naturality.
\end{proof}

Besides these two Lemmas, we will make use of a result dating back to Milnor in \cite{Milnor}. He looked at functions $f:\Q/\Z\to \C$ satisfying the Kubert identity
	$$f(x)=m^{s-1}\sum_{k=0}^{m-1}f\left(\frac{x+k}{m}\right)$$
for fixed $s$ and all integers $m$ and all $x\in\Q/\Z$. Identifying $\Q/\Z$ with the roots of unity in $\C$ (by $x\mapsto e^{2\pi ix}$), we can write $f(x)=L(e^{2\pi ix})$ and the Kubert identity becomes
	$$L(\zeta^m)=m^{s-1}\sum_{\xi^m=1}L(\zeta\xi).$$
Milnor proved the following result:

\begin{Theorem}[Milnor, 1983]
Let $\Q/\Z$ have the quotient topology. The space of continuous functions $f:\Q/\Z\to \C$ satisfying the Kubert identity is two-dimensional and splits into two one-dimensional spaces, the first of which contains all the functions with $L(\zeta)=L(\bar\zeta)$ and the second, the ones with $L(\zeta)=-L(\bar\zeta)$.
\end{Theorem}

It is an unproven conjecture by Milnor (also in \cite{Milnor}) that one can drop the continuity assumption on $f$ and the theorem would still hold in the same form.\\

\textit{Proof of Theorem \ref{basecase}.}
To any higher twisted torsion invariant $\tau$ we get for any $n$-th root of unity a coefficient $s_1(\tau,\zeta)$ in
	 $$\tau(S^1(\lambda)/n,\F_\zeta)\stackrel{\textnormal{def }s_1}=s_1(\tau,\zeta)ch_{2k}(\lambda)\in H^{2k}(\C\P^\infty;\R)\cong \R.$$
Identifying $\Q/\Z$ with the roots of unity, we get a function $f_{\tau}:\Q/\Z\to \R$ defined by 			
	$$f_\tau(\zeta):=\frac 1 {n^k}s_1(\tau,\zeta),$$
where $\zeta^n=1.$ This is well defined, since by a previous Lemma we have
	$$\tau(S^1(\lambda)/(nm),\F_\zeta)=m^k\tau(S^1(\lambda)/n,\F_\zeta),$$
so $f_\tau(\zeta)$ is by construction independent from the choice of $n$ with $\zeta^n=1$.\\
Our goal is to show that this satisfies the Kubert identity and then to use Milnor's result to prove our theorem. But for this, $f_\tau$ needs to be continuous, a fact which we cannot prove, but must assume. Therefore we need the following last axiom:

\begin{Axiom}[Continuity]
For any twisted torsion invariant, the function $f_\tau:\Q/\Z\to \R$ constructed above is continuous.
\end{Axiom}

\begin{proof}[Continuation of the proof] Now we calculate for $\zeta\in \Q/\Z$ with $\zeta^n=1$ using the two Lemmas from above:
	\begin{eqnarray*}
		f_\tau(\zeta^m)ch_{2k}(\lambda)	&	=	&	\frac1{n^k}\tau(S^1(\lambda)/nm,\F_{\zeta^m})\\
																		&	=	&	\frac1{n^k}\sum_{\xi^m=1}\tau(S^1(\lambda)/nm,\F_{\xi\zeta})\\
																		&	=	&	m^k\sum_{\xi^m=1}f_\tau(\xi\zeta)ch_{2k}(\lambda).
	\end{eqnarray*}
So $f_\tau$ satisfies the Kubert identity (with $s=k+1$) for any $\tau$.\\
Now we need to check which one of the two one-dimensional subspaces of the space of functions satisfying the Kubert identity $f_\tau$ is in. For this, we note that the change of representation from $\zeta$ to $\bar\zeta$ represents a change of orientation in the fiber. So it corresponds to a map $g:\C\P^\infty\to \C\P^\infty,$ giving $g_*:\pi_1S^1\to\pi_1 S^1$ as multiplication by $-1$. Using   that $\pi_1 S^1(\lambda)/n\cong \Z/n$, we get the following commutative diagram relating  the exact sequence of the homotopy groups of the fibration $S^1\into S^1(\lambda)/n\to \C\P^\infty$ to itself under $g_*$:
	$$\xymatrix
			{
				\pi_2\C\P^\infty\ar[d]^{g_*}\ar[r]^{\cdot n}	&	\Z\ar[d]^{-1}\ar@{->>}[r]	&	\Z/n\ar[r]\ar[d]^{g_*}	&	0\\
				\pi_2\C\P^\infty\ar[r]^{\cdot n}							&	\Z\ar@{->>}[r]						&	\Z/n\ar[r]							&	0.
			}
	$$
From this one can see that $g_*:\pi_2\C\P^\infty\to \pi_2\C\P^\infty$ is multiplication by $-1$. Since $\C\P^\infty$ is simply connected, $g_*$ is also multiplication by $-1$ in homology of degree $2$. Since $\C\P^\infty$ is an Eilenberg-MacLane space, $g^*$ must be multiplication by $-1$ on degree $2$ cohomology and therefore multiplication by $(-1)^k$ on degree $2k$ cohomology.\\
This yields 
	$$f_\tau(\zeta)=(-1)^kf_\tau(\bar\zeta)$$
for any $\tau$ with degree $2k.$ So $f_\tau$ is in one specific one-dimensional subspace of the space of functions satisfying the Kubert identity for any torsion invariant $\tau$ of degree $2k,$ and therefore we have for an arbitrary torsion invariant $\tau$ and the Igusa-Klein torsion $\tau^{IK}$
	$$f_\tau=af_{\tau^{IK}}$$
for a certain $a\in \R$. This translates to
	$$\tau(S^1(\lambda)/n,\F_\zeta)=a\tau^{IK}(S^1(\lambda)/n,\F_\zeta)$$
for any root of unity $\zeta$ and proves the theorem.
\end{proof}

\begin{Remark} Given that Milnor's conjecture on continuity is correct, we do not need the functions $f_\tau$ to be continuous anymore and therefore can drop the whole continuity axiom.
\end{Remark}

\section{The Difference Torsion}

Given a twisted torsion invariant $\tau$, we can now form the twisted difference torsion
	$$\tau^\delta:=\tau-a\tau^{IK}-bM,$$
where the scalars $a$ and $b$ are the ones from the statement (and $b$ is $0$ if the torsion has degree $4l+2$).
Our goal now is to show that $\tau^\delta(E,\F)=0$ for every bundle $E\to B$ with every local coefficient system $\F$ on $E$ and rationally simply connected $B$. Clearly, $\tau^\delta$ is a twisted torsion invariant.

\subsection{Lens Space Bundles}

The goal now is to show that the difference torsion is zero on every linear odd dimensional lens space bundle $L^{2N+1}_n\into E^{2N+1}_n\to B$ with local coefficient system $\F$ on $E_n^N$. We already know from the base case that the difference torsion is zero on every $S^1$-bundle. Furthermore, if we take an $S^l$-bundle with $l>1$ or disk bundle, we know that the fundamental group of the fiber is trivial and it therefore admits no non-constant local system. So the twisted difference torsion on these bundles is always given by the non-twisted difference torsion. But the non-twisted difference torsion is zero everywhere as K. Igusa showed in \cite{Igusa1}. From this we get the following Lemma:

\begin{Lemma}
For the difference torsion $\tau^\delta$ associated with any higher twisted torsion invariant, we have
	$$\tau^\delta(E,\F)=0$$
for any disk or sphere bundle $E\to B$ with local system $\F$ on $E$.
\end{Lemma}

At first we will prove:

\begin{Lemma} The difference torsion is $0$ on any linear odd dimensional lens space bundle $L^{2N+1}_n\into E^{2N+1}_n\to B$. By linear we mean that it is covered by a linear sphere bundle $S^{2N+1}\into \tilde E^{2N+1}\to B.$
\end{Lemma}

\begin{proof} The covering sphere bundle $\tilde E^{2N+1}$ is a subbundle of an $N+1$-dimensional complex vector bundle. By the splitting principle, it suffices to look at the direct sum of $N+1$ complex line bundles. The sphere bundle will become the fiberwise joint of the circle bundles associated with the line bundles:
	$$S^1*\ldots *S^1\into \tilde E^1_1*\ldots*\tilde E^1_{N+1}\to B.$$
Now we have 
	\begin{eqnarray*}
		L^{2N+1}_n	&	\cong	&	(S^{2N-1}*S^1)/n\\
								&	\cong	&	(S^{2N-1}/n)*(S^1/n)\\
								&	=			&	(S^{2N-1}\times D^2)/n\cup_{(S^{2N-1}\times S^1)/n}(D^{2N}\times S^1)/n.
	\end{eqnarray*}
Fiberwise, this gives us 
	$$E^{2N+1}_n=H_n^{2N-1}\cup H_n^1,$$
where $H_n^{2N-1}\to B$ is a $(S^{2N-1}\times D^2)/n$ bundle and $H_n^1\to B$ is a $(D^{2N}\times S^1)/n$ bundle, both meeting along their common vertical boundary given by an $(S^{2N-1}\times S^1)/n$-bundle $G_n$. The $n$-action is hereby given by the  simultaneous action on each component of the crossproducts. While the $n$-action on any disk is not free, the simultaneous action will guarantee that it is free on the crossproduct. We can restrict every local coefficient system $\F$ on $E_n^{2N+1}$ to $H_n^{2N-1},$ $H_n^1$ and $G_n$ and use the additivity axiom.\\  
Now we will continue the proof by induction. We know that the difference torsion is $0$ on every $L^1_n\cong S^1$-bundle. Let us then assume that the difference torsion is $0$ on any linear $L^{2N-1}_n$-bundle with any representation of the fundamental group. Given a linear $L^{2N+1}_n$-bundle $E^{2N+1}_n\to B$ with local coefficient system $\F$, the construction above yields
	$$\tau^\delta(E_n^{2N+1},\F)=\tau^\delta(H_n^{2N-1},\F_{|H_n^{2N-1}})+\tau^\delta(H_n^1,\F_{|H_n^1})-\tau^\delta(G_n,\F_{|G_n}).$$
We have non-trivial fibrations $D^2\into (S^{2N-1}\times D^2)/n\to L_n^{2N-1}$, $D^{2N}\into (D^{2N}\times S^1)/n\to L_n^1$ and $S^1\to (S^{2N-1}\times S^1)/n\to L_n^{2N-1}.$ The first of these splits the bundle $H_n^{2N-1}$ in the following manner:
	$$
		\xymatrix
			{
				D^2\ar@{^{(}->}[r]	&	(S^{2N-1}\times D^2)/n\ar[rr]\ar@{^{(}->}[d]	&	&		L_n^{2N-1}\ar@{^{(}->}[ld]\\
				D^2\ar@{^{(}->}[r]	&	H_n^{2N-1}\ar[r]\ar[d]							&	J_n\ar[ld]\\
									&	B,
			}
		$$
where $J_n\to B$ is an $L_n^{2N-1}$-bundle and $H_n^{2N-1}\to J_n$ is a $D^2$-bundle. Since $D^2$ is contractible, we get a local system $\F_J$ on $J_n$ the pull-back of which to $H^{2N-1}_n$ is isomorphic to $\F_{|H^{2N-1}_n}.$ Now we can use the geometric transfer and the fact that we already determined the difference torsion to be $0$ on $L_n^{2N-1}$- and $D^2$-bundles to show 
	$$\tau^\delta(H_n^{2N-1},\F_{|H_n^{2N-1}})=\chi(D^2)\tau(J_n,\F_J)+\tr_B^{J_n}(\tau_{J_n}(H^{2N-1}_n,\F_|H_n^{2N-1}))=0.$$
A similar argument holds for $H_n^1$ and $G_n,$ and this completes the proof.
\end{proof}


\subsection{Difference Torsion as a Fiber Homotopy Invariant}

In this section, we will prove that the difference torsion $\tau^\delta$ is a fiber homotopy invariant.  By this we mean that for any two bundles $F_1\into E_1\to B$ and $F_2\into E_2\to B$ and fiber homotopy equivalence $f:E_1\to E_2$ with local coefficient systems $\F_2$ on $E_2$ and $f^*\F_2\cong\F_1$ on $E_1$, we have
	$$\tau^\delta(E_1,\F_1)=\tau^\delta(E_2,\F_2)\in H^{2k}(B;\R).$$
This section will greatly rely on the construction of the equivariant Hatcher examples from \cite{Hatcherconstruction}. We will especially use some techniques involving \hcob bundles, for a basic depiction of which the leader is also referred to \cite{Hatcherconstruction}.\\

First we show the following lemmas:

\begin{Lemma}\label{disklemma} For any linear disk bundle $D\stackrel{q}{\to} E$ and any  bundle pair $(E,\partial_0)\to B$ with local coefficient system $\F$ we have
	$$\tau^\delta_B(D,\partial_0,q^*\F)=\tau^\delta_B(E,\partial_0,\F),$$
where we pull the system up to $D$ and $\partial_0D=q^{-1}\partial_0 E$ as usual.
\end{Lemma}
\begin{proof} By geometric transfer we have
	$$\tau^\delta_B(D,\partial_0,q^*\F)=\tau^\delta(E,\partial_0, \F)+\tr_B^E(\tau^\delta_E(D,q^*\F))$$
and $\tau^\delta_E(D,q^*\F)=0$ because $D$ is a disk bundle over $E$.
\end{proof}

We will now need to prove three subsequent lemmas before we can prove the fiber homotopy invariance.

\begin{Remark} We saw in section 5.1 that it is enough to look at local systems that induce holonomy covers with cyclic transformation group. So we will always assume that.
\end{Remark}

\begin{Lemma} Let $B$ be a rationally simply connected space. Then for arbitrarily large integers $N$ the difference torsion $\tau^\delta$ is zero on any \hcob bundle of $L_n^{2N-1}$ for a given $n.$
\end{Lemma}

\begin{proof} Since $B$ is simply connected, all local systems on an \hcob bundle of $L_n^{2N-1}\times D^M$ inducing an $n$-fold cyclic holonomy are isomorphic to the local systems of the form $\F_\zeta,$ where $\zeta$ is an $n$-th root of unity. We will now fix such a $\zeta.$\\

We will follow K. Igusa closely in his discussion of the untwisted version of this proof \cite{Igusa1}. By the stability of higher torsion we can view the difference torsion as a map
	$$\tau^\delta(\underline{\quad},\F_\zeta):[B,B\P(L^{2N-1}_n)]=[B,B(\lim_\to\curlyC(L_n^{2N-1}\times D^M))]\to H^*(B;\R)$$
sending an \hcob bundle $h\to B$ to $\tau^\delta(h,\F_\zeta).$ We can give the set $[B,B\P(L_n^{2N-1})]$ a group structure by the fiberwise gluing together of the \hcobs as explained in \cite{Igusa1}. From the additivity properties of higher twisted torsion it follows that $\tau^\delta(\underline{\quad},\F_\zeta)$ is a group homomorphism. So it is enough to give rational generators of $[B,B\P(L^{2N-1}_n)]$ and show that the difference torsion is zero on these generators.\\

For $N$ large enough ($N>>\dim B$) we have
	$$[B,B\P(L^{2N-1}_n)]=[B,\H(L^{2N-1}_n)]\cong [B,\H(B\Z/n)].$$
In \cite{Hatcherconstruction} we define the twisted Hatcher maps $\Delta^i:G_n/U\to \H(B\Z/n)$ and the main theorem thereof tells us (because $\dim B<<N$) that thereby the space $\Q\otimes [B,\H(B\Z/n)]$ is spanned by various Hatcher constructions (also defined in \cite{Hatcherconstruction}) of one non trivial vector bundle $\xi$ over $B$ with fiber homotopically trivial sphere bundle. The calculations in \cite{Hatcherconstruction} only rely on the axioms and ensure that the difference torsion of these Hatcher constructions is zero. 
\end{proof}

\begin{Lemma} Let $N$ be an arbitrarily large integer and $E\to B$ a bundle with local system $\F$ on $E$ inducing an $n$-fold cyclic holonomy covering. Then we have $\tau^\delta(E,\F)=0$ if there is a fiber homotopy equivalence
	$$\xymatrix{
		E\ar[r]^(.3){\sim}\ar[d]	&	L_n^{2N-1}\times B\ar[d]\\
		B\ar[r]^=						&	B.
		}
	$$ 
\end{Lemma}

\begin{proof} Denote the fiber homotopy equivalence $H:E\to L_n^{2N-1}\times B.$ We can take the product of $L_n^{2N-1}\times B$ with a large dimensional disk $D^M$ and make $H$ into an embedding
	$$\bar H:E\stackrel{\sim}{\into} D^M \times L_n^{2N-1}\times B.$$
Then we can take a tubular neighborhood of $\bar H(E)\subseteq D^M\times L_n^{2N-1}\times B$ to get a codimension $0$ embedding of an $M$-dimensional disk bundle $D(E)$ over $E$
	$$G:D(E)\stackrel\sim\into D^M\times L_n^{2N-1}\times B.$$
Then $(D^M\times L_n^{2N-1}\times B)\backslash G(D^\circ(E))$ is an \hcob bundle of $L_n^{2N-1}\times S^{M-1}$ over $B$ and by the additivity axiom its difference torsion is given by
	\begin{eqnarray*}
		\tau^\delta((D^M\times L_n^{2N-1}\times B)\backslash G(D^\circ(E)),\F)	&	+	&\\
		\tau^\delta(D(E),\F)-\tau^\delta(S(E),\F)													&	=	&	\tau^\delta(D^M\times L_n^{2N-1}\times B,\F)=0,
	\end{eqnarray*}
since the last bundle is trivial. $S(E)$ hereby is the sphere bundle given as the vertical boundary of $D(E).$ We can use the transfer axiom to show that	
	$$\tau^\delta(D(E),\F)=\tau^\delta(E,\F)$$
and given that $M$ is odd
	$$\tau^\delta(S(E),\F)=0$$
because the difference torsion is zero on any disk and sphere bundles. So it suffices to show that the difference torsion is zero on any \hcob bundle of $L_n^{2N-1}\times S^{M-1}$ over $B$ for arbirtrarily large $N.$ Such a bundle can easily be reduced to an \hcob bundle of $L_n^{2N-1}$ without changing its torsion: Let $H\to B$ be an \hcob bundle of $L_n^{2N-1}\times S^{M-1}.$ We can embedd $S^{M-1}\times I$ as a tubular neighborhood of $S^{M-1}$ into $D^M$ and thereby get 
	$$H\supseteq L_n^{2N-1}\times S^{M-1}\times B\into L_n^{2N-1}\times D^M\times 1\times B\subseteq L_n^{2N-1}\times D^M\times I\times B$$
and we can define the \hcob bundle of $L_n^{2N-1}\times D^M$ (and thereby of $L_n^{2N-1}$ by stability) 
	$$H':=H\cup_{L_n^{2N-1}\times S^{M-1}\times B} L_n^{2N-1}\times D^M\times I\times B.$$
We calculate using the relative additivity properties of higher torsion for any local system $\F$ in $L_n^{2N-1}$ extended naturally to $H$, $H'$ and $L_n^{2N-1}\times D^M\times I\times B$
	\begin{eqnarray*}
		\tau^\delta(H',\F)	&	=	&	\tau^\delta(H',L_n\times D^M\times 0\times B,\F)\\
					&	=	&	\tau^\delta(H,L_n^{2N-1}\times S^{M-1}\times B, \F)\\
					&	+	&	\tau^\delta(L_n^{2N-1}\times D^m\times I\times B,L_n^{2N-1}\times D^m\times 0\times B,\F)\\
					&	=	&	\tau(H,\F)
	\end{eqnarray*}

With this construction on \hcob bundles the proof now follows from the previous Lemma.
\end{proof}

\begin{Lemma} The difference torsion $\tau^\delta$  is $0$ on any  bundle pair $(E,\partial_0)\to B$ the fibers $(F,\partial_0)$ of which  are h-cobordisms and have a local system $\F$ inducing a cyclic $n$-fold holonomy covering.
\end{Lemma}
\begin{proof} First observe that since $(F,\partial_0)$ is an h-cobordism, $F\simeq \partial_0 F$ and therefore the local system is already given by a local system on $\partial_0 F.$ The holonomy covering of $\partial_0 F$ gives rise to a map $\partial_0 F\to K(\Z/n,1)\simeq L^\infty_n,$ and therefore we get a map $\partial_0F\to L_n^N$ and a map $\partial_0 E\into B\times L_n^N$ for a large $N$. We can choose a local system $\F_L$ on $L^N_n$ that restricts to $\F_{|\partial_0 E}$ on $\partial_0 E.$\\
Now choose a smooth fibered embedding of $\partial_0 E$ into $B\times D^N$ for some large $N$ (we use the same $N$ as before, meaning that we will eventually have to enlarge it). Let $\partial_0D$ be the normal disk bundle of $\partial_0 E$ in $B\times D^N$. Since $E$ is fiber homotopy equivalent to $\partial_0 E$, because it is a h-cobordism bundle, the linear disk bundle $\partial_0 D$ extends to a linear disk bundle $D$ over $E$ with projection $q:D\to E.$ By a previous lemma we have $\tau_B^\delta(E,\partial_0,\F)=\tau_B^\delta(D,\partial_0,q^*\F).$ Thus it is enough to show that $\tau_B^\delta(D,\partial_0,q^*\F)=0$. Now we can also look at the normal disk bundle of $\partial_0 E$ in $B\times L_n^N$ and get a map $\partial_0 D\to B\times D^2\times L_n^N,$ where $\partial_0 D$ is embedded into $B\times S^1\times L_n^N\subseteq \partial^v (B\times D^2\times L_n^N).$ We can extend the local system $\F_L$ to $\F_{D^2\times L}$ and $D^2\times L_n^N$ by pulling it back along the projection onto the second component. Now we look at the union
	$$B\times D^2\times L^N_n\cup_{\partial_0 D} D.$$
This will be a lens space bundle with local system induced by the gluing together of the local coefficient systems on the two components. This lens space bundle will also be fiber homotopically equivalent to a trivial bundle. We can assume without loss of generality that $N$ is odd with the  previous lemma and additivity of higher torsion we get
	\begin{eqnarray*}
		\tau_B^\delta(D,\partial_0,q^*\F)&=&\tau_B^\delta(D,q^*\F)-\tau_B^\delta(\partial_0,q^*\F|{\partial_0 D})+\tau_B(B\times D^2\times L^N_n,\F_{D^2\times L})\\
																			&=&\tau^\delta((B\times L_n^N)\cup D,\F_L\cup q^*\F)=0.
\end{eqnarray*}

\end{proof}

\begin{Theorem}
The difference torsion $\tau^\delta$ is a fiber homotopy invariant of  smooth bundle pairs with local systems.
\end{Theorem}
\begin{proof}

Suppose that $(E_1,\partial_0)$ and $(E_2,\partial_0)$ are  smooth bundle pairs over $B$ which are fiber homotopy equivalent and have matching local coefficient systems $\F_1$ and $\F_2.$ We want to show that
	$$\tau^\delta(E_1,\partial_0, \F_1)=\tau^\delta(E_2,\partial_0,\F_2).$$
If we replace $(E_2,\partial_0)$ by a large dimensional disk bundle $(D_2,\partial_0)$ (with local system $q_2^*\F_2$ for $q_2:D_2\to E_2$), we can approximate the fiber homotopy equivalence by a fiberwise smooth embedding
	$$g:(E_1,\partial_0)\to (D_2,\partial_0).$$
Let $D_1$ be the normal disk bundle of $E_1$ in $D_2$ (with local system $q_1^*\F_1$ for $q_1:D_1\to E_1$). Then the closure of the complement of  $D_1$ in $D_2$ is a  fiberwise $h$-cobordism with local coefficient given by restricting $q_2^*\F_2$, giving an $h$-cobordism bundle $H$, which has trivial $\tau^\delta$ by the previous Lemma. Using the relative additivity we therefore get 
	$$\tau^\delta(D_2,\partial_0,q^*\F_2)=\tau^\delta(H,\partial_1 D_1,(q_2^*F_2)_{|H})+\tau^\delta(D_1,\partial_0,q_1^*\F_1)=\tau^\delta(D_1,\partial_0,q_1^*\F_1).$$
By Lemma \ref{disklemma} this yields $\tau^\delta(E_1,\partial_0,\F_1)=\tau^\delta(E_2,\partial_0,\F_2).$
\end{proof}

\begin{Remark}
Since $\tau^\delta$ is a fiber homotopy equivalence, it is well defined on any fibration $(Z,C)\to B$ with fiber $(X,A)$ and local system $\F$ on $X$ which is smoothable in the sense that it is fiber homotopy equivalent to a smooth bundle pair $(E,\partial_0)$ with compact manifold fiber $(F,\partial_0)$.
\end{Remark}

\section{Triviality of the Difference Torsion}

\subsection{Lens Space Suspensions}

In the following, we want to define for a space $F$ with local system $\F$ inducing an $n$-fold holonomy cover a suspension construction, which respects the local system. Let us recall that the usual suspension $\Sigma F$ is defined by the homotopy push-out
	$$\xymatrix{
			F\ar[r]\ar[d]	&	S^\infty\ar[d]\\
			S^\infty\ar[r]			&	\Sigma F.
		}
	$$
Since $S^\infty$ is contractible, we know that $\pi_1\Sigma F=0$, and therefore this construction cannot give us a non-constant local system on $\Sigma F$. Now we make the following definition:

\begin{Definition}[lens space suspension] Let $F$ be a topological space with local system $\F$ on $F$ inducing an $n$-fold holonomy cover $\tilde F\to F$ with finite cyclic transition group. The cover gives us a mapping $F\to L_n^{2N}$ for a large $N\in \N$ (because $L^\infty_n\cong K(\Z/n, 1)$). Using this map, we can define the lens space suspension $\Sigma_n F$ as the homotopy push-out
	$$\xymatrix{
			F\ar[r]\ar[d]	&	L_n^{2N}\ar[d]\\
			L_n^{2N}\ar[r]		&	\Sigma_n F.
		}
	$$
\end{Definition}

\begin{Remark}
We will drop $N$ from the notation and consider it to be very large.
\end{Remark}

We have the earlier introduced local systems $\F_\zeta$ on $L^{2N}_n$ for an $n$-th root of unity $\zeta.$ By choosing the map $i:F\to L_n^{2N}$ properly, we can assume that $\F=i^*\F_{e^{2\pi i/n}}.$ So we get a local system $\Sigma\F=\F_{e^{2\pi i/n}}\cup_\F\F_{e^{2\pi i/n}}$ on $\Sigma_n F.$ From this we get the holonomy covering $\tilde{\Sigma_n F}\stackrel{n}{\to} \Sigma_n F$; but we also have the holonomy covering $\tilde{F}\stackrel{n}{\to} F$. These two covering spaces are related by the following Lemma:

\begin{Lemma}\label{coveringlemma} In the setting above, we have
	$$\tilde{\Sigma_n F}\simeq \Sigma\tilde{F}$$
in low degrees (smaller than $2N$).
\end{Lemma}

\begin{proof} By using the coverings $\tilde{F}\stackrel{n}{\to} F$ and $S^{2N}\stackrel{n}{\to}L^{2N}_n$, we get the diagram
	$$\xymatrix{
			\tilde{F}\ar[rr]\ar[dd]\ar[dr]^n	&	&	S^{2N}\ar'[d][dd]\ar[dr]^n\\
					&	F\ar[rr]\ar[dd]	&	&	L_n^{2N}\ar[dd]\\
			S^{2N}\ar'[r][rr]\ar[dr]^n							&	&	\Sigma(N)\tilde{F}\ar@{.>}[dr]\\
					&	L_n^{2N}\ar[rr]	&&	\Sigma_n F,
			}
	$$
where $\Sigma(N)\tilde F:=S^{2N}\cup_{\tilde F}S^{2N}$ is homotopy equivalent in low degrees to $\Sigma\tilde F$ because there is an $2N$-connected map 
	$$\Sigma(N)\tilde F\to \Sigma \tilde F\cong S^\infty\cup_{\tilde F} S^\infty.$$
The universal property of the push-out gives the map $\Sigma(N)\tilde{F}\to \Sigma_n F$. Being covering spaces, $\tilde F$ and $S^{2N}$ both admit an $n$-action, and the map $\tilde F\to S^{2N}$ will be $n$-equivariant. Therefore the push-out $\Sigma(N)\tilde F$ will also admit an $n$-action, and from this it follows that the map $\Sigma(N)\tilde F\to \Sigma_n F$ is an $n$-fold covering, and therefore we have $\tilde{\Sigma_n F}\simeq \Sigma\tilde{F}$ in low degrees.
\end{proof}

For the usual suspension, it is well known that $H_{k+1}(\Sigma F;\R)\cong H_k(F; \R)$ for all $k\in \N$. For the lens space suspension this becomes:

\begin{Lemma}\label{homologylemma}
For every topological space $F$ with local system inducing an $n$-fold holonomy covering, we have for $k\geq 1$
	$$H_{k+1}(\Sigma_n F;\R)\cong H_k(F; \R).$$
\end{Lemma}

\begin{proof} Using the Mayer Vietoris sequence for the defining push-out of the lens space suspension, we get
	\begin{eqnarray*} 
	\to 	& H_{k+1}(L_n^{2N}; \R)\oplus H_{k+1}(L_n^{2N}; \R)	& \to H_{k+1}(\Sigma_n F;\R)\to H_k(F;\R)\\
	\to		&	H_{k}(L_n^{2N}; \R)\oplus H_{k}(L_n^{2N}; \R)\to
	\end{eqnarray*}
The fact that $L_n^{2N}$ is rationally homologically trivial now yields the desired isomorphism.
\end{proof}

Furthermore, we know for the usual suspension that $\pi_m^S(F)\otimes \R\cong \bar{H}_m(F;\R)$, where $\pi_m^S(F):=\pi_m(\colim_k\Omega^k\Sigma^k F)$ denotes the stabilized homotopy group. This becomes:

\begin{Lemma}
If $k\in \N$ is large enough, and $F$ is a space with local system inducing an $n$-fold holonomy covering, we have an isomorphism
	$$\pi_{m+k}(\Sigma_n^k F)\otimes \R\cong \bar{H}_{m+k}(\Sigma^k\tilde F;\R)$$
for $m+k<N$. 
\end{Lemma}
\begin{proof}
We get the $n$-fold holonomy covering $\tilde{F}\to F.$ Using Lemma \ref{coveringlemma} several times, we get in low degrees
	$$\tilde{\Sigma_n^k F}\simeq\Sigma(\tilde{\Sigma_n^{k-1}F})\simeq\ldots\simeq\Sigma^k\tilde{F}.$$
Thus we have for $N>m+k>1$
	\begin{eqnarray*}
		\pi_{m+k}(\Sigma_n^kF)\otimes \R	&	\cong	&	\pi_{m+k}(\Sigma^k\tilde{F})\otimes\R\\
			&\stackrel{k\textnormal{ large}}{\cong}	&	\pi_m^S(\tilde F)\otimes\R\\
																			&	\cong	&	\bar{H}_m(\tilde F;\R)\\
																			&	\cong	&	\bar{H}_{m+k}(\Sigma^k \tilde F;\R).
	\end{eqnarray*}
\end{proof}

\begin{Remark}
Although we require $k$ to be large in the last Lemma, it does not depend on $N$ at all, meaning that we can still choose $N$ to be much larger than $k$.
\end{Remark}

We will need the following Definition and Proposition:

\begin{Definition} A topological space $F$ is called simple if $\pi_1 F$ is abelian and acts trivially on every $\pi_i F$ for $i\geq 2.$
\end{Definition}

\begin{Proposition} Let $F$ be a path connected, simple space and $\tilde F\stackrel{n}\to F$ an $n$-fold Galois covering. Then the transition group $\Z/n$ will act trivially on $H_*(\tilde F;\R)$
\end{Proposition}
\begin{proof}
Let $\{F^l\}$ be the Postnikow tower for $F;$ that is a sequence of spaces with $\lim_l F^l\cong F$ and $\pi_i F\cong \pi_i F$ for $0\leq i\leq l$ and $\pi_i F^l\cong 0$ for $i>l.$ Since we have $\pi_1 F^l\cong \pi_1 F$ for every $l>0,$ we have $n$-fold coverings $\tilde F^l\stackrel n \to F^l.$ We will prove by induction that $\Z/n$ acts trivially on $H_*(\tilde F^l;\R).$ The sequence $\{\tilde F^l\}$ will clearly provide a Postnikow tower for $\tilde F$, and since the real homology of the stages of a Postnikow tower stabilizes in every degree, this will prove the proposition.\\
To start the induction we look at $F^1\simeq K(\pi_1 F,1)$, which will only have the first homotopy group $\pi_1 F^1\cong \pi_1 F$. The covering $\tilde F\stackrel n \to F$ gives a map $\alpha:\pi_1 F\to \Z/n.$ Using this, we see that the covering $\tilde F^1\stackrel n \to F^1$ will be an Eilenberg-MacLane space:
	$$\tilde F^1\simeq  K(\ker \alpha, 1).$$
The group $\Z/n$ acts trivially on $\ker \alpha\subseteq \pi_1 F$ because $\pi_1 F$ is abelian, and therefore $\Z/n$ acts trivially on $\tilde F^1\simeq K(\ker \alpha, 1)$ and $H_*(\tilde F^1;\R).$ This starts the induction.\\
Now assume that $\Z/n$ acts trivially on $H_*(\tilde F^{l-1};\R)$ with $l>1$. We have the fibration
	$$K(\pi_l \tilde F, l)\to \tilde F^l\to \tilde F^{l-1}.$$
Since we know $\pi_l F\cong \pi_l \tilde F,$ the group $\Z/n$ will act trivially on $\pi_l \tilde F$ and thereby also trivially on $K(\pi_l \tilde F, l)$ and $H_*(K(\pi_l \tilde F, l); \R)$. By induction assumption it must also act trivially on	
	$$H_i(\tilde F^{l-1}; H_k (K(\pi_l(\tilde F),l);\R)),$$
and thereby it acts trivially on the whole Leray Serre spectral sequence for the fibration $K(\pi_l \tilde F, l)\to \tilde F^l\to \tilde F^{l-1}.$ From this it follows that $\Z/n$ acts unipotently on $H_*(\tilde F_l;\R),$ and since $\R[\Z/n]$ is semisimple, this includes that $\Z/n$ acts trivially on $H_*(\tilde F;\R).$
\end{proof}

From this we get the following important Corollary:

\begin{Corollary}\label{simplefiber} If $F$ is a simple topological space with local system inducing an $n$-fold holonomy covering $\tilde F\stackrel n\to F$, then we have
	$$H_l(\Sigma(N)^k \tilde F;\R)\cong H_l(\Sigma_n^k F;\R)$$
for all $l<2N$.
\end{Corollary}

\begin{proof} Since $F$ is simple, the group $\Z/n$ will act trivially on $H_*(\tilde F;\R)$. It is well known that this implies 
	$$H_*(F;\R)\cong H_*(\tilde F;\R)$$
(See for example Proposition 3.G.1 in Hatcher's book \cite{Hatcher}). By using Lemma \ref{homologylemma} and the $2N$-connected map $\Sigma(N)\tilde F\to \Sigma \tilde F$ we get
	$$H_*(\Sigma(N)^k \tilde F;\R)\cong H_{*-k}(\tilde F;\R)\cong H_{*-k}(F;\R)\cong H_*(\Sigma_n^k F;\R)$$
up to degree $2N.$
\end{proof}

Now we are turning back to bundles. For a fiber bundle $F\into E\to B$ with local system $\F$ on $F$ inducing a finite cyclic $n$-fold holonomy covering, we get a fiberwise map $E\to B\times L_n^{2N}$ and can use this to define the fiberwise lens space suspension as the push-out
	$$\xymatrix{
			E\ar[r]\ar[d]	&	B\times L_n^{2N}\ar[d]\\
			B\times L_n^{2N}\ar[r]	&	\Sigma_{n,B} E.
		}
	$$
It is easy to see that $\Sigma_{n,B}E\to B$ is a bundle with fiber $\Sigma_n F$ and as before we get a local system $\Sigma \F$ on $\Sigma_{n,B}E.$ We have the following Lemma:

\begin{Lemma}\label{torsionofsuspension}
The bundle $\Sigma_{n,B}E$ is smoothable if $E$ was smoothabl and we have
	$$\tau^\delta(E,\F)=-\tau^\delta(\Sigma_{n,B}E,\Sigma\F).$$
\end{Lemma}

\begin{proof} The bundle $\Sigma_{n,B}E$ is smoothable since it is the fiberwise push-out of smooth bundles along a smoothable bundle. Additivity gives
	$$\tau^\delta(\Sigma_{n,B}E,\Sigma\F)=\tau^\delta(B\times L^{2N}_n\cup_EB\times L_n^{2N},\Sigma \F)=-\tau^\delta(E,\F).$$
\end{proof}

\subsection{Reducing the Homology of the Fiber}

We now attempt to make the fiber of a bundle $F\into E\to B$ with a local system on $F$, simply connected base $B$, and simple fiber $F$ rationally homologically trivial without changing the difference torsion. For this we first need two Lemmas. In the following, let $N$ always be an arbitrarily large integer.

\begin{Lemma}\label{Lemma1}
Suppose $F\into E\to B$ is a  fibration with local system $\F$ on $F$ inducing a finite cyclic $n$-fold holonomy covering. Let $m\in \N$ denote the largest integer for which $\bar{H}_m(F;\R)\neq 0$. Suppose that we have $H_l(F;\R)\cong H_l(\tilde F;\R)$ for $0<l<m+\dim B$. Suppose further that $m$ is odd and let $\alpha$ be a map
	 $$\alpha:B\times L_n^{m}\to E$$
with the following properties: On each fiber we have $\alpha^*\F\cong \F_\zeta$ for some $n$-th root of unity $\zeta$ and $\alpha_*:\bar H_m(L_n^m;\R)\to \bar H_m(E;\R)$ is non-trivial. Then if we look at the bundle 
	$$E_1=E\cup_{B\times L_n^{m+k}} B\times L_n^{2N}$$
with fiber $F_1$ with local system $\F_1:=\F\cup_{\F_\zeta}\F_\zeta$ and corresponding covering $\tilde F_1\stackrel n\to F_1$, we have $\dim_\R H_*(F_1;\R)<\dim_\R H_*(F;\R)$ and $H_l(F_1;\R)\cong H_l(\tilde F_1;\R)$ for $0<l<m+\dim B.$
\end{Lemma}

\begin{proof}
Assume that we have a map $\alpha:B\times L_n^{m}\to E$ such that the induced map
	$$\alpha_*:\bar{H}_{m}(L_n^{m+k};\R)\to \bar{H}_{m}(\Sigma_n^k F;\R)$$
is non-trivial. Then the homology of the fiber $F_1$ will be given by the Mayer Vietoris sequence as 
	$$H_{m}(L_N^{m};\R)\stackrel {\alpha_*}\to H_{m}( F;\R)\oplus 0\to H_{m}(F_1;\R)\to 0$$
and thereby we have $\dim_\R H_*(F_;\R)<\dim_\R H_*(F;\R).$\\
To show that this $F_1$ will satisfy the second property, we look at the covering of $\alpha$
	$$\tilde\alpha: B\times S^{m}\to \tilde E,$$
which will also be non-trivial on homology. We get the following cubic diagram, where the back face is a push-out covering a push-out, forming the front face:
	$$\xymatrix{
			&	B\times S^{m}\ar[dl]_n\ar[rr]^{\tilde\alpha}\ar[dd]|!{[dl];[dr]}\hole	&&\tilde E\ar[ld]_n\ar[dd]\\
		B\times L_n^{m}\ar[rr]^\alpha\ar[dd]	&&	E\ar[dd]\\
			&	B\times S^{2N}\ar[rr]|!{[dl];[dr]}\hole\ar[dl]_n	&&	\tilde E_1\ar[dl]_n\\
		B\times L_n^{2N}\ar[rr]	&&	E_1.
		}
	$$
Since we have $H_{l}(F;\R)\cong H_{l}(\tilde F;\R)$ for $0<l<m+\dim B$ by assumption, and the only non-vanishing real homology group of the source is in the odd degree $m$, the homology up to degree $2N$ of $F_1$ and $\tilde F_1$ will be copies of the homology of $F$ and $\tilde F$ in any low degree except for $m.$ In degree $m$ we have 
	$$H_{m}(F_1;\R)\cong H_{m}(F;\R)/\im(\alpha)\cong H_{m}(F;\R)/\R$$
and 
	$$H_{m}(\tilde F_1;\R)\cong H_{m+k}(\tilde F;\R)/\im(\tilde\alpha)\cong H_{m}(\tilde F;\R)/\R.$$
Therefore we get
	$$H_{l}(F_1;\R)\cong H_{l}(\tilde F_1;\R)$$
for $0<l<m+\dim B.$\\
\end{proof}

\begin{Lemma}\label{Lemma2} Suppose $F\into E\to B$ is a  fibration with simply connected base $B$ and local system $\F$ on $F$ inducing a finite cyclic $n$-fold holonomy covering. As before  let $m\in \N$ denote the largest integer for which $\bar{H}_m(F;\R)\neq 0$ and suppose that we have $H_l(F;\R)\cong H_l(\tilde F;\R)$ for $0<l<m+\dim B$. Then there exists an integer $k\in \N$ and a map 
	$$\alpha:B\times L_n^{m+k}\to \Sigma_{n,B}^kE$$
such that $\alpha^*\Sigma^k\F\cong \F_\zeta$ for some $n$-th root of unity $\zeta$ and $\alpha_*:\bar H_{m+k}(L_n^{m+k};\R)\to \bar H_{m+k}(\Sigma_n^kF;\R)$ is non-trivial.
\end{Lemma}
\begin{proof} Note that in the following, $m$ and $n$ are fixed, already determined integers, whereas $k$ is an arbitrarily large integer bounded by the arbitrarily large integer $N$. Furthermore $m+k$ must be odd, such that $L^{m+k}_n$ has a non-vanishing rational homology group in degree $m+k$, but we can choose $k$ in such a way that this is satisfied.\\
Such a map $\alpha$ will correspond to a section $s$ of the bundle 
	$$\Map(L_n^{m+k}, F)\into \Map(B\times L_n^{m+k},\Sigma_{n,B}^k E)\to B,$$
which is a homologically non-trivial map in each fiber. In this context $\Map(B\times L_n^{m+k},\Sigma_{n,b}^E)$ will always denote the space of fiberwise maps between $B\times L_n^{m+k}$ and $\Sigma_{n,B}^kE$. We will construct this section using obstruction theory. Let $B_l$ denote the $l$-skeleton of $B$. Firstly, we will give $s_1:B_1\to \Map(B\times L_n^{m+k},\Sigma_{n,B}^k E)$. By the choice of $m$ we have a non-zero element:
	$$\tilde\gamma\in \bar H_{m+k}(\Sigma(N)^k \tilde F;\R)\cong\bar{H}_{m+k}(\Sigma_n^k  F;\R)\cong \bar{H}_{m}(F;\R).$$
Since the reduced homology is isomorphic to rationalized stabilized homotopy, we can view $\tilde\gamma$ as element of  $\pi_{m+k}(\Sigma(N)^k\tilde F)\otimes \R$, if $k$ is large enough.  Now choose a representative $\tilde\alpha_1: S^{m+k}\to \Sigma(N)^k \tilde F$ of $\tilde{\gamma}$. The map $\tilde \alpha_1$ will clearly be non-trivial on homology.\\ 
Our goal is now to modify $\tilde\alpha_1$ to $ \tilde\alpha: S^{m+k}\to \Sigma(N)^k\tilde F$ such that it covers an $\alpha:L_n^{m+k}\to \Sigma_n^kF$. Since $H_{m+k}(\Sigma(N)^k \tilde F;\R)\cong H_{m+k}(\Sigma_n^k F;\R)$, $\alpha$ will be non-trivial on homology. Furthermore the covering will ensure that $\alpha^*\F\cong \F_\zeta$ for some $n$-th root of unity $\zeta$. To begin with, we have from the last lens space suspension an inclusion
	$$i:L_n^{m+k}\into \Sigma_n^k F$$
trivial on homology. This will be covered by a homologically trivial equivariant inclusion
	$$\tilde i:S^{m+k}\into \tilde {\Sigma_n^k F}.$$
The idea now is to take a small disk $D^{m+k}$ in $S^{m+k}\subseteq \tilde{\Sigma_n^k F}$ and connect it to the image $\tilde \alpha_1(S^{m+k}).$ Then we can map $S^{m+k}$ to this new image instead and this map will be non-trivial on homology because $\tilde \alpha_1$ is non-trivial on homology. To make it equivariant we do the same construction equivariantly to every disk $p^i D^{m+k}$ in the orbit of $D^{m+k}$ under the $\Z/n$ action on $S^{m+k}$. Hereby $p\in \Z/n$ denotes a generator. This is illustrated in figure 1.
	
\begin{figure}
\label{modify}

\begin{picture}(200,200)
\setlength{\unitlength}{1cm}
\put(0,0){\includegraphics[width=1.00\textwidth]{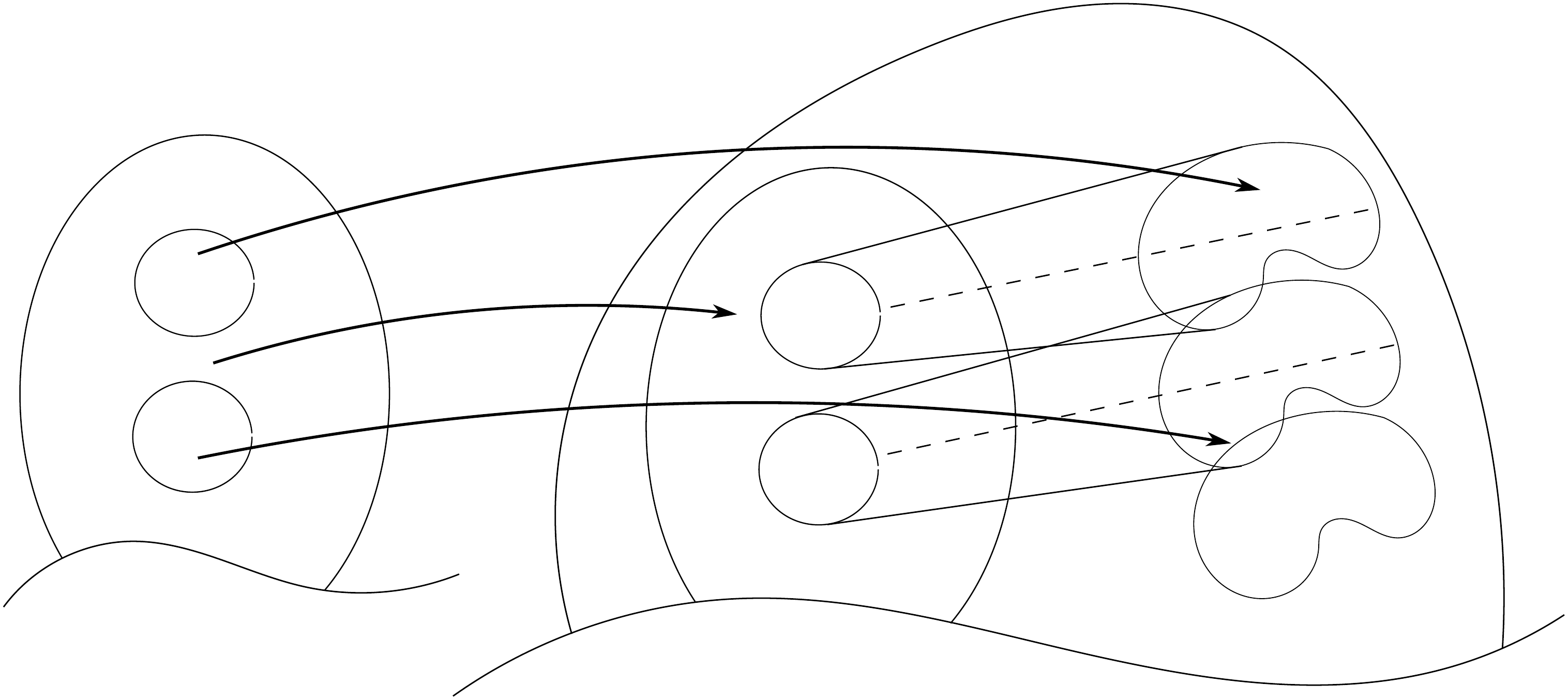}}
\put(4.5,4.2){$\tilde i$}
\put(0.55,3.2){$S^{m+k}$}
\put(0.4,5){\line(3,-2){1.4}}
\put(0.1,5.2){$D^{m+k}$}
\put(1,1){\line(1,2){0.8}}
\put(0.65,0.6){$pD^{m+k}$}
\put(10.3,6.1){$\tilde{\Sigma_n^kF}$}
\put(13.2,5){\line(3,1){1.3}}
\put(13.2,3.8){\line(5,1){1.8}}
\put(14.6,5.4){$\tilde\alpha_1(S^{m+k})$}
\put(15.1,4.2){$p\tilde\alpha_1(S^{m+k})$}
\end{picture}
\caption{Modifying the inclusion $\tilde i:S^{m+k}\into \tilde{\Sigma_n^kF}.$}
\end{figure}	
\noindent The formal construction is the following: Choose a small disk $D^{m+k}\subseteq S^{m+k}$. By doing this in a slightly bigger disk, we can modify the inclusion such that it factorizes 
	$$D^{m+k}\to *\into \tilde{\Sigma_n^kF}.$$
Using $D^{m+k}/\partial D^{m+k}\simeq S^{m+k}$, we can glue in $\tilde \alpha_1$ and modify the inclusion again so that it factorizes
	$$D^{m+k}\stackrel {\tilde\alpha_1}\to \tilde{\Sigma_n^k F}.$$
Now let $p\in \Z/n$ be a generator. If we make $D^{m+k}$ small enough, it will not intersect with any of the $p^iD^{m+k}\subseteq S^{m+k}$ for $0<i<n$. Doing the same construction to every $p^i D^{m+k}$ using $p^i\tilde \alpha_1,$ we can modify the inclusion to a map
	$$\tilde \alpha:S^{m+k}\to \tilde{\Sigma_n^kF},$$
which will clearly be  $n$-equivariant and thus cover a map
	$$\alpha: L_n^{m+k}\to \Sigma_n F.$$
The corresponding rationalized homotopy class of $\tilde \alpha$ in $\pi_{m+k}(\Sigma_n^k F)\otimes \R$ will be given by
	$$[\tilde \alpha]=[\tilde \alpha_1]+p[\tilde\alpha_1]+\ldots+p^{n-1}[\tilde\alpha_1]=n[\tilde\alpha_1]\neq 0,$$
since $\pi_1 F$ acts trivially on 

\begin{eqnarray*}
	\pi_{m+k}\tilde{\Sigma_n^kF}\otimes\R	&	\cong	&	\pi_{m+k}(\Sigma(N)^k\tilde F)\otimes\R\cong H_{m+k}(\Sigma(N)^k\tilde F; \R)\\
				&	\cong	&	H_m(\tilde F; \R)\cong H_m(F;\R)
\end{eqnarray*}
(otherwise the map $H^m(F;\R)\into H^m(\tilde F;\R)$ would not be an isomorphism and thereby $H_m(F;\R)$ would not be isomorphic to $H_m(\tilde F;\R)$ either). So $\alpha$ will be non-trivial in rational homology.\\
With this we can define $s_0:B_0\simeq *\to \Map(B\times L_n^{m+k},\Sigma_{n,B}^k E)$ non-trivial in the homology of the fiber. Since $B$ is simply connected, this section, defined over a point of $B$, can be extended to a section $s_1:B_1\to\Map(B\times L_n^{m+k},\Sigma_{n,B}^k E)$.\\
Let us now continue inductively. Suppose we already have a section $s_l:B_l\to \Map(B\times L_n^{m+k},\Sigma_{n,B}^k E)$ with $1\leq l <\dim B$. By restriction, we will get sections
	$$s_{l,i}:B_l\to \Map(B\times L_n^{i},\Sigma_{n,B}^kE).$$
Let us first extend $s_{l,1}$ to $s_{l+1,1}:B_{l+1}\to \Map(B\times L_n^{1},\Sigma_n^kE)$: This depends on the obstruction class
	$$\theta(s_{l},1)\in H^{l+1}(B,B_l;\pi_l(\Map(L_n^1,\Sigma_n^kF)))\cong H^{l+1}(B,B_l;\pi_{l+1}(\Sigma_n^kF)),$$
because $L_n^1\simeq S^1$. So $\theta(s_{l,1})$ is rationally trivial, if $k$ is large enough (larger than $l+1$). This is enough to extend $s_{l,1}$ as, for example, K. Igusa showed in the non-twisted version of this Lemma in \cite{Igusa1}.\\
We now want to extend $s_{l+1,1}$ to $s_{l+1,2}$ relative to $s_{l,2}$. For this we look at the cofibration sequence
	$$L_n^1\into L_n^2\to S^2,$$
which gives us the fibration sequence
	$$\Omega^2(\Sigma_n^kF)\into \Map(B\times L^2_n,\Sigma_{n,B}^k E)\to \Map(B\times L^1_n,\Sigma_{n,B}^k E).$$
From this we get the following commutative diagram:
	$$\xymatrix
		{
				&																		\Omega^2(\Sigma_n^kF)\ar@{^{(}->}[d]\\
			B_l\ar@{^{(}->}[d]\ar[r]^(.25){s_{l,2}}	&	\Map(B\times L_n^2,\Sigma_{n,B}^kE)\ar[d]\\
			B_ {l+1}\ar@{.>}[ur]^(.4){s_{l+1,2}}\ar[r]^(.25){s_{l+1,1}}	&	\Map(B\times L_n^1,\Sigma_{n,B}^kE),
		}
	$$
where the right column is a fibration sequence. Consequently the extension from $s_{l+1,1}$ to $s_{l+1,2}$ depends on the obstruction class
	$$\theta(s_{l,1})\in H^{l+1}(B, B_l; \pi_{l}(\Omega^2(\Sigma^k_n F)))\cong H^{l+1}(B, B_l; \pi_{l+2}(\Sigma^k_n F)),$$
which is, again, rationally trivial for large $k$.\\
Now assume that we have already constructed $s_{l+1,i}$ with $i\in \N$ even. Next, look at the cofibration
	$$L_n^i\into L_n^{i+2}\to M(\Z_n,i),$$
where
	$$M(\Z_n,i):=cof(S^{i}\stackrel n \to S^{i})$$
is the Moore space. Directly from the definition of the Moore space, we get that $\pi_l(\Map(M(\Z_n, i),X))$ is finite for any space $X$. Using the fibration
	$$\Map(M(\Z_n,i),\Sigma_n^kF)\into \Map(B\times L^{i+2}_n,\Sigma_{n,B}^k E)\to \Map(B\times L^i_{n,B}\Sigma_n^k E),$$
the commutative diagram
	$$\xymatrix
		{
				&																		\Map(M(\Z_n,i),\Sigma_n^kF)\ar@{^{(}->}[d]\\
			B_l\ar@{^{(}->}[d]\ar[r]^(.25){s_{l,i+2}}	&	\Map(B\times L_n^{i+2},\Sigma_{n,B}^kE)\ar[d]\\
			B_ {l+1}\ar@{.>}[ur]^(.4){s_{l+1,i+2}}\ar[r]^(.25){s_{l+1,i}}	&	\Map(B\times L_n^i,\Sigma_{n,B}^kE),
		}
	$$
tells us that extending $s_{l+1,i}$ to $s_{l+1, i+2}$ depends on the obstruction class
	$$\theta(s_{l+1,i})\in H^{l+1}(B,B_l;\pi_l(\Map(M(\Z_n, i),\Sigma_n^kF))),$$
which is rationally trivial.\\
Using this inductively, we get $s_{l+1,k+m-1}$. To extend this to $s_{l+1,k+m}=s_{l+1}$, we use again the cofibration sequence
	$$L_n^{k+m-1}\into L_N^{k+m}\to S^{k+m}$$
and the induced fibration sequence
	$$\Omega^{k+m}(\Sigma_n^kF)\into \Map(B\times L^{k+m}_n,\Sigma_{n,B}^k E)\to \Map(B\times L^{k+m-1}_n,\Sigma_{n,B}^k E)$$
and the commutative diagram
	$$\xymatrix
		{
				&	&																	\Omega^{k+m}(\Sigma_n^kF)\ar@{^{(}->}[d]\\
			B_l\ar@{^{(}->}[d]\ar[rr]^(.25){s_{l,k+m}}	&	&\Map(B\times L_n^{k+m},\Sigma_{n,B}^kE)\ar[d]\\
			B_ {l+1}\ar@{.>}[urr]^(.4){s_{l+1,k+m}}\ar[rr]^(.3){s_{l+1,k+m-1}}	&	&\Map(B\times L_n^{k+m-1},\Sigma_{n,B}^kE),
		}
	$$
making the obstruction class
	$$\theta(s_{l+1,k+m-1})\in H^{l+1}(B,B_l;\pi_{k+m+l}(\Sigma_n^kF)).$$
However, if $k$ is large enough, we have
	\begin{eqnarray*}
		\pi_{k+m+l}(\Sigma_n^kF)\otimes \R	&	\cong	& \pi_{k+m+l}(\Sigma(N)^k\tilde F)\otimes \R\\
																				&	\cong &	\bar H_{k+m+l}(\Sigma(N)^k\tilde F;\R)\\
																				&	\cong &	H_{k+m+l}(\Sigma_n^k F;\R)\\
																				&	\cong	&	H_{m+l}(F;\R)\cong 0
	\end{eqnarray*}
by assumption because $m+l<m+\dim B$. This guarantees that we can extend $s_{l+1,k+m-1}$ to $s_{l+1}$ and completes the proof.
\end{proof}

\begin{Lemma} Let $F\into E\to B$ be a  fibration with simply connected base $B$ and local system $\F$ on $F$ inducing a finite cyclic $n$-fold holonomy covering. Suppose further that $F$ is simple. Then there exists a bundle $F'\into E'\to B$ with local coefficient system $\F'$ on $F',$ where $F'$ is rationally homologically trivial such that
	$$\tau^\delta(E,\F)=\pm \tau^\delta(E',\F').$$
\end{Lemma}
\begin{proof} Let again $m$ be the largest integer such that $H_m(F;\R)$ is non-trivial. Since $F$ is simple we get by Corollary \ref{simplefiber} $H_*(F;\R)\cong H_*(\tilde F;\R)$ and we can use Lemma \ref{Lemma2} to get a bundle map
	$$\alpha:B\times L_n^{m+k}\to \Sigma_{n,B}^kE$$
for a integer $k$ non-trivial on the $m+k$-th homology. By Lemma \ref{coveringlemma} the $n$-fold covering of $\Sigma_n^kF$ is given in low degrees by $\Sigma(N)^k\tilde F$. Since both $\Sigma_n^k$ and $\Sigma(N)^k$ only shift rational homology up by $k$ degrees we have 
	$$H_l(\Sigma_n^kF;\R)\cong H_l(\Sigma(N)^kF;R)\cong H_l(\tilde{\Sigma_n^kF};\R)$$
for all $0<l<m+k+\dim B.$ Furthermore the highest non-trivial homology group of $\Sigma_n^k F$ is in degree $m+k$ and we also have 
	$$\dim_\R H_*(\Sigma_n^kF;\R)=\dim_\R H_*(F;\R).$$
Now we can apply the construction of Lemma \ref{Lemma1} to get a bundle $F_1\into E_1\to B$ such that
	$$\dim_\R H_*(F_1,\R)<\dim_\R H_*(\Sigma_n^kF;\R)=\dim_\R H_*(F;\R).$$
By definition of $E_1$, and since the torsion of trivial bundles is zero, we get with additivity and Lemma \ref{torsionofsuspension}
	$$\tau^\delta(E_1,\F_1)=\tau^\delta(\Sigma_{n,B}^kE,\Sigma^k\F)=(-1)^k\tau^\delta(E,\F).$$
Since Lemma \ref{Lemma1} also guarantees that $H_l(F_1;\R)\cong H_l(\tilde F_1;\R)$ for $0<l<m+k+\dim B$ we now can repeat this process and decrease the dimension of the rational homology until we will get the bundle $F'\into E'\to B$ with local system $\F'$ on $F$ such that 
	$$\tau^\delta(E,\F)=\pm \tau^\delta(E',\F')$$
and $F'$ is rationally homologically trivial.
\end{proof}

\begin{Remark}
As a consequence of this Lemma it suffices to only determine $\tau^\delta$ on bundles with rationally trivial fiber. We will proof in the next section that the difference torsion will always be zero on these bundles and this will conclude the proof of the main theorem.
\end{Remark}

\subsection{Triviality on Fibers with Trivial Real Homology}

\begin{Lemma} We have $\tau^\delta(Z,\F)=0$ for any torsion invariant, smoothable bundle $X\into Z\to B$ with $\bar H_*(X;\R)=0,$ simply connected base $B$ and local system $\F$ inducing an $n$-fold holonomy covering.
\end{Lemma}
\begin{proof}
By taking the fiberwise lens space suspension, we may assume that the map $Z\to B$ has a section. We choose a smooth bundle $E\to B$ fiber homotopy equivalent to $Z$. By embedding $E$ into $B\times \R^m$ for large $m$ and taking a tubular neighborhood, we may assume that the vertical tangent bundle of $E$ is trivial and that the fiber is a compact $m$-manifold $M$ with boundary $\partial M$ embedded in $\R^m$ so that $M\simeq X$. Also, the image of the section $B\to E$ will have a neighborhood, which is a trivial disk bundle $D\cong B\times D^m$. Let $\tilde M$ be the $n$-fold universal cover of $M$ (using lens space suspension we can ensure that $\pi_1M\cong \Z/n$). This gives raise to the fiberwise covering $\tilde E\stackrel{n}{\to} E$. The disk $D\subseteq M$ lifts to $n$ disks $\tilde D\subseteq \tilde M$. The bundle $\tilde E\to B$ is classified by a map $B\to B\Diff(\tilde M \rel \tilde D)$; but since it is a covering of an $M$-bundle, this map restricts to $f:B\to B\Diff_n(\tilde M\rel \tilde D)$, where $B\Diff_n(\tilde M\rel \tilde D)$ is the space of $n$-equivariant diffeomorphisms of $\tilde M$ which may permute the components of the lifted disk $\tilde D$. By construction $\tilde E$ will be the quotient of $F\times G$, where $F:=\Diff_n(\tilde M \rel \tilde D)$ and $G:=f^*E\Diff_n(\tilde M \rel \tilde D)$ under the diagonal $F$-action. Putting this together we get the following commutative diagram:
	$$\xymatrix{
					&	F\times G\ar[ld]\ar[ddd]\ar[rr]	&					&	F\times E\Diff_n(\tilde M \rel \tilde D)\ar[ld]\ar[ddd]\\
		\tilde E\ar[d]_n\ar[ddr]\ar[rr]|!{[ur];[ddr]}\hole	&&	\tilde E_{univ}\ar[d]_n\ar[ddr]\\
		E\ar[dr]\ar[rr]|!{[u];[dr]}\hole|!{[uur];[dr]}\hole	&		&	E_{univ}\ar[dr]\\
			&	B\ar[rr]	&&	B\Diff_n(\tilde M \rel \tilde D).
		}$$
$\tilde E_{univ}$ is hereby the quotient of $F\times E\Diff_n(\tilde M \rel \tilde D)$ under the diagonal action. This will admit an $n$-action, and $\tilde E\to \tilde E_{univ}$ will be $n$-equivariant. Therefore it covers the universal bundle $E_{univ}$ and $\tau^\delta(E,\rho)$ is the pull back of
	$$\tau^\delta:=\tau^\delta(E_{univ},\rho)\in H^{2k}(B\Diff_n(\tilde M \rel \tilde D);\R).$$
Thus it suffices to show that $\tau^\delta=0$.\\
Since $B$ is simply connected, we know that the classifying map will factorize 
	$$B\to B\Diff_{0,n}(\tilde M \rel \tilde D)\to B\Diff_{n}(\tilde M \rel \tilde D),$$
where $B\Diff_{0,n}(\tilde M \rel \tilde D)$ is the identity component. Ergo we consider
	$$\tau^\delta\in H^{2k}(B\Diff_{0,n}(\tilde M \rel \tilde D);\R).$$
Since the identity component will only contain maps that leave a certain base point fixed, we have
	$$B\Diff_{0,n}(\tilde M \rel \tilde D)\cong B\Diff_0(M \rel D).$$
Now choose $M_0$ to be $M$ without an open collar neighborhood of $\partial M$. We can choose $M_0$ in such a way that $D\subseteq M_0$. Naturally, we get a covering $\tilde M_0\stackrel n \to M_0$. Let  $\Diff_n(\tilde M\rel \tilde M_0)$ be the space of $n$-equivariant diffeomorphisms of $\tilde M$, which leave $\tilde M_0$ fixed (and do not permute any components thereof). For this we have
	$$\Diff_n(\tilde M\rel \tilde M_0)\subseteq \Diff_n(\tilde M\rel \tilde D)$$
and get a map 
	$$\psi:\pi_0\Diff_n(\tilde M\rel \tilde M_0)\to \pi_0 \Diff_n(\tilde M\rel \tilde D).$$
The kernel of $\psi$ will be the set of connected components of $\Diff_n(\tilde M\rel \tilde M_0)$, which map into the identity component $\Diff_{0,n}(\tilde M\rel \tilde D)$. Taking only these components, we get a space with inclusion
	$$\Diff_{\ker \psi,n}(\tilde M\rel \tilde M_0)\subseteq \Diff_{0,n}(\tilde M\rel\tilde D),$$
and from this we get a map
	$$p:B\Diff_{\ker \psi,n}(\tilde M\rel \tilde M_0)\to B\Diff_{0,n}(\tilde M\rel\tilde D).$$
The pull-back of $\tau^\delta$ along $p$ will be the torsion class of an $M$-bundle containing a trivial $M_0$-bundle as a fiberwise deformation retract. Since $\tau^\delta$ is a fiber homotopy invariant and trivial on trivial bundles, the pull-back $p^*\tau^\delta$ will be $0$.\\
Therefore it suffices to show that
	$$p^*:H^{2k}(B\Diff_{0,n}(\tilde M\rel\tilde D);\R)\to H^{2k}(B\Diff_{\ker \psi,n}(\tilde M\rel \tilde M_0);\R)$$
is injective. To do this, we will show that $p$ is rationally $2k$-connected: \\
Since the maps in $\Diff_{n}(\tilde M\rel \tilde M_0)$ fix a base point, we have 
	$$\Diff_{n}(\tilde M\rel \tilde M_0)\cong \Diff(M\rel M_0).$$
Now we can form the space $\Diff_{\ker\psi}(M\rel M_0)$ by taking the corresponding connected components such that $\Diff_{\ker \psi}(M\rel M_0)=\Diff_{\ker \psi, n} (\tilde M\rel \tilde M_0)$. We can view the map $p$ as
	$$p:B\Diff_{\ker\psi}(M\rel M_0)\to B\Diff_0(M\rel D).$$
We will show that this is rationally $2k$-connected. At first we will get the following sequence we want to show to be a fibration:
	$$\Diff_{\ker\psi}(M\rel M_0)\into \Diff_0(M\rel D)\stackrel \pi\to Emb_0(M_0,M\rel D),$$
where $Emb_0(M_0,M\rel D)$ is the identity component of the space of embeddings $M_0\into M$ fixing $D$. Here, $\pi$ is simply the restriction map. By the isotopy extension theorem, $\pi$ will be surjective. From this it follows from a theorem of Cerf (\cite{Cerf}, Appendix) that $\pi$ is a fibration with fiber being the preimage of any point. We easily see that $\pi^{-1}(id)\cong \Diff_{\ker \psi}(M\rel M_0),$ and therefore the sequence above is a fibration. Applying the $B$-functor, we get another fibration:
	$$Emb_0(M_0,M\rel D)\into B\Diff_{\ker\psi}(M\rel M_0)\stackrel p \to B_0(M\rel D).$$
So we just need to show that
	$$\pi_iEmb_0(M_0,M\rel D)\otimes \R\cong 0\quad \textnormal{for }0<i<2k.$$
When the dimension $m$ we embedded $M$ in is large, the homotopy dimension of $M_0$ will be much smaller than $m$. Therefore, by transversality, we have that the embedding space is homotopy equivalent in low degrees to the corresponding space of immersions
	$$\pi_i Imm_0(M_0,M\rel D)\cong \pi_i Emb_0(M_0,M\rel D).$$
By immersion theory, $Imm_0(M_0,M\rel D)$ is homotopy equivalent to the space of all pointed homotopy equivalences $M_0\to M$ and the space of all pointed maps $M_0\to O(m)$:
	$$\pi_i Imm_0(M_0,M\rel D)\cong \pi_ih-eq_*(M_0,M)\oplus \pi_i(\Map_*(M_0,O(m))).$$
Since $M_0\simeq M$ and the space of pointed homotopy equivalences is the identity component of the space of pointed maps, we have
	$$\pi_i h-eq_*(M_0,M)\cong \pi_i(\Map_*(M,M)).$$
The theorem will follow from the next Proposition.
\end{proof}

\begin{Proposition}
Let $M$ be a pointed space with $\bar H_*(M;\R)=0$. Then 
	$$\pi_i \Map_*(M,X)\otimes \R\cong 0$$
for any $i>0$ and any pointed finite CW complex $X$.
\end{Proposition}
\begin{proof}
We can take the Postnikow tower $\{X^n\}$ of $X$. This is a sequence of spaces $X^n$ such that $X=\lim_n X^n$, and we have a fibration
	$$K(\pi_nX,n)\into X^n\to X^{n-1}$$
for all $n.$ We will now use induction on $n$ to show for all $i$
	$$\pi_i(\Map_*(M,X^n))\otimes \R\cong 0.$$
We can use the fibration sequence
	$$\Map_*(M,K(\pi_nX,n))\into \Map_*(M,X^n)\to \Map_*(M, X^{n-1})$$
to see that the homotopy group $\pi_i \Map_*(M,X^n)$ will stabilize for every $i$ as $n$ gets larger. So proving $\pi_i(\Map_*(M,X^n))\otimes \R\cong 0$ for every $n$ will prove the Proposition.\\
Since $X^0$ is just a point, the start of induction is trivial.\\
Now suppose that we know $\pi_i(\Map_*(M,X^{n-1}))\otimes \R\cong 0$ for all $i$. Again, we have the fibration
	$$\Map_*(M,K(\pi_nX,n))\into \Map_*(M,X^n)\to \Map_*(M, X^{n-1}).$$
The long exact sequence of homotopy groups thereof gives
	$$\pi_i\Map_*(M,K(\pi_nX,n))\to \pi_i\Map_*(M,X^n)\to \pi_i\Map_*(M, X^{n-1}).$$
Since $\pi_i\Map_*(M, X^{n-1})$ is rationally trivial, it is enough to show that $\pi_i\Map_*(M,K(\pi_nX,n))$ for all $i$ as well: we have for given $i$
	\begin{eqnarray*}
		\pi_i\Map_*(M,K(\pi_nX,n))	&	\cong	&	[S^i, \Map_*(M,K(\pi_n X,n))]\\
															&	\cong	&	[\Sigma^iM,K(\pi_nX,n)]\\
															&	\cong	&	\bar H^{n-i}(M;\pi_nX),
	\end{eqnarray*}
and this is rationally trivial by assumption.
\end{proof}

\bibliography{mybib1}{}
\bibliographystyle{plain}

\end{document}